\documentclass[11pt,english,oneside]{amsart}
\usepackage[T1]{fontenc}
\usepackage[latin9]{inputenc}
\usepackage{geometry}
\usepackage{amssymb}
\usepackage{color}

\usepackage{graphicx,color}
\usepackage{amsmath, amssymb, graphics}

\usepackage{graphicx}

\makeatletter
\numberwithin{equation}{section} 
\numberwithin{figure}{section} 
  \@ifundefined{theoremstyle}{\usepackage{amsthm}}{}
  \theoremstyle{plain}
  \newtheorem{thm}{Theorem}[section]
  \theoremstyle{plain}
  \newtheorem{cor}[thm]{Corollary}
  \theoremstyle{plain}
  
  \theoremstyle{remark}
  \newtheorem{rem}[thm]{Remark}
  \theoremstyle{remark}
  
  \theoremstyle{plain}
  \newtheorem{lem}[thm]{Lemma}
   \newtheorem{mydef}{Definition}
   \newtheorem{example}{Example}

\usepackage{geometry}

\def\bfR#1{{\bf R}^#1}

\def\com#1{ \hbox{#1}}

\def\e{\hbox{\rm e}}

\smallskip
\def\<{{\langle }}
\def\>{{\rangle }}

\makeatother

\usepackage{babel}

\def\bfR#1{{\bf R}^#1}

\def\com#1{ \quad\hbox{#1}\quad}

\def\e{\hbox{\rm e}}

\smallskip
\def\<{{\langle }}
\def\>{{\rangle }}

\makeatother

\begin{document}

\title{Constant-speed ramps}

\author{ Oscar M. Perdomo }

\date{\today}

\curraddr{Department of Mathematics\\
Central Connecticut State University\\
New Britain, CT 06050\\}

\email{ perdomoosm@ccsu.edu}

\begin{abstract}

It is easy to show that if the kinetic coefficient of friction between a block and a ramp is  $\mu_k$ and this ramp is a straight line with slope $-\mu_k$ then, this block will move along the ramp with constant speed. A natural question to ask is the following: Besides straight lines, are there more shapes of ramps such that a block will go down on the ramp with constant speed. Here we classify all the possible shape of these ramps and, {\bf surprisingly } we will show that the planar ramps can be parametrized in term of elementary functions: trigonometric function, exponential functions and their inverses. They provide basic examples of explicitly parametrized arc-length parameter curves. A video explaining the main results in this paper can be found in one of the following URLs:

 http://www.youtube.com/watch?v=iBrvbb0efVk  

http://www.ccsu.edu/cf\_media/index.cfm?g=871

\end{abstract}

\maketitle

\section{Introduction}

A basic computation shows that given a kinetic coefficient of friction $\mu_k$, the angle of a linear ramp can be adjusted so that a block on this ramp will move down with constant speed. In order to cancel the gravity force, the friction force and the normal force, it is enough to take a ramp given by a line with slope $-\mu_k$, this is, we must take the inclination of the ramp to be $\theta_0$ with $\tan(\theta_0)=\mu_k$.  See Figure \ref{line}

\begin{figure}[ht]
\centerline{\includegraphics[width=8cm,height=5cm]{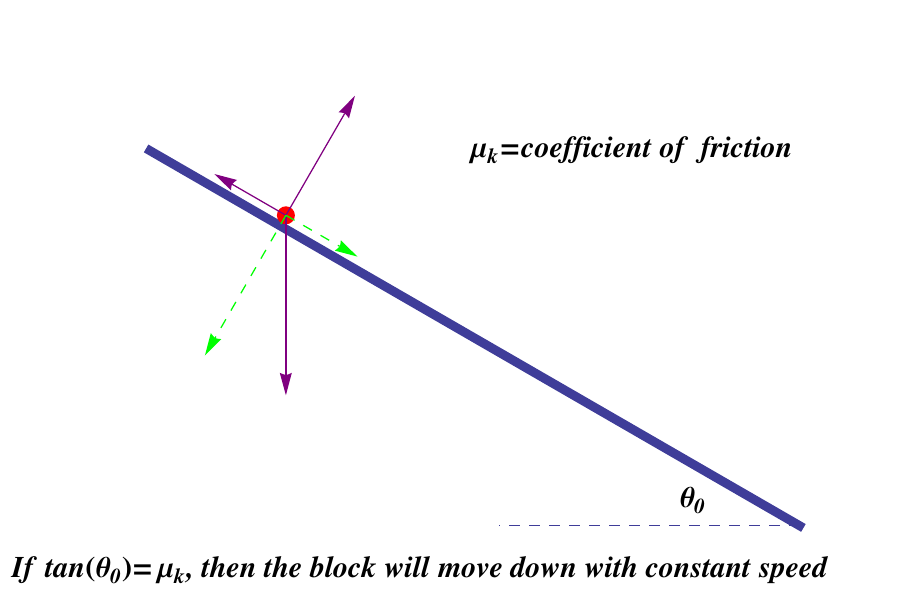}}
\caption{A block sliding down under the effect of gravity and the friction force. This figure shows the three forces acting on the block. The dashed forces are just a decomposition of the gravity force.}
\label{line}
\end{figure}

Let us assume that we want a block to move down with speed $v_0$ and we know that the kinetic coefficient of friction is $\mu_k$. In this note we classify the shapes of all possible ramps on which this block will move down with speed $v_0$. Here is how you would built this {\it constant-speed ramp}. Given $v_0$ and $\mu_k=\tan(\delta)$, let us define $a=\frac{g}{v_0^2 \sin\delta}$, ($g=9.81\, \frac{m}{s^2}$) and 

\begin{eqnarray}\label{alpha}
\alpha(s)=\left( s+\frac{1}{a}  \ln(1+\e^{-2as}), \frac{2}{a}\, \hbox{arccot}(\e^{-as}) \right)
\end{eqnarray}

This curve $\alpha$ has an U shape with two horizontal asymptotes, one at $y=0$ and the second one at $y=\frac{\pi}{a}$. See Figure
\ref{alpha1}. We have:
\begin{center}
{\it  If we clockwise rotate the curve $\alpha$ an angle $\delta$ (see Figure \ref{alpha23}), then, the highest point divides this rotated curve in two ramps on which a block can move with constant speed $v_0$ under the assumption that the kinetic coefficient of friction is $\mu_k=\tan(\delta)$}  
\end{center}
  
\begin{figure}[ht]
\centerline{\includegraphics[width=10cm,height=3cm]{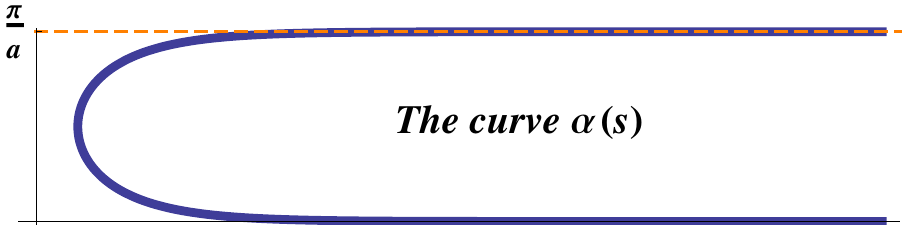}}
\caption{This is the graph of the curve $\alpha$, it has two horizontal asymptotes that are separated $\frac{\pi}{a}$. The parameter $s$ in the parametrization provided in the definition of $\alpha$ is arc-length.}
\label{alpha1}
\end{figure}

\begin{figure}[ht]
\centerline{\includegraphics[width=9cm,height=5cm]{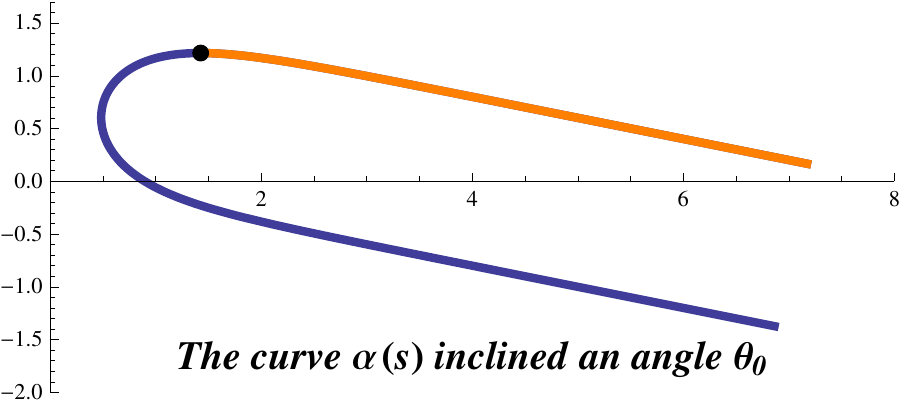}}
\caption{Graph of the two  {\it constant-speed ramps}. At the highest point, the block with speed $v_0$ must be placed on the top of the ramp to the right and it must be placed upside down on the ramp to the left. }
\label{alpha23}
\end{figure}

The sequence of pictures at the end of the paper shows how the motion on the ramps takes place when the desire speed is $5\, \frac{m}{s}$ and $\mu_k=0.5$. The picture also displays the three forces acting on the block under the assumption that the mass is $1$ Kg. On the highest part of ramp the normal is zero, the normal in the upper part of the ramp is pointing down, the bock does not fall down due to its speed. 

In section  \ref{ramp2d} we provide a first proof for the classification of two dimensional ramps. In section \ref{deframp} we provide a formal definition of ramp in order to state the result as a mathematical theorem. This section may be omitted. Section \ref{3dramps} deals with 3-dimensional ramps.

The author would like to express his gratitude to Frank Gould, Roger Vogeler, Nidal Al-Masoud and Clifford Anderson for their valuable comments on this work.

\section{A first approach}\label{ramp2d}
Let us find all the possible curves with the property that a block, sliding down on it, will move with constant speed $v>0$ under the assumption that the kinetic coefficient of friction is $\mu=\tan(\delta)$ for some constant $\delta$ between $0$ and $\frac{\pi}{4}$ radians. Let us start by assuming that this curve is parametrized by arc-length, this is, if $\alpha(s)=(x(s),y(s))$ denotes such a curve, then $|\alpha^\prime(s)|=1$.  Under this assumption we can assume that for a smooth function $\theta(s)$ we have

$$ x^\prime(s)=\cos(\theta(s))\com{and} y^\prime(s)=\sin(\theta(s))$$

The function $\theta(s)$ will help us to describe the curve $\alpha$. It is clear that the vector $n(s)=(-\sin(\theta(s)),\cos(\theta(s)))$ is a unit vector perpendicular to $\alpha^\prime(s)$ and the chain rule give us that $\alpha^{\prime\prime}(s)=\theta^\prime(s)\, n(s)$. Figure \ref{vectors} shows these two vectors

\begin{figure}[ht]
\centerline{\includegraphics[width=4.6cm,height=4.2cm]{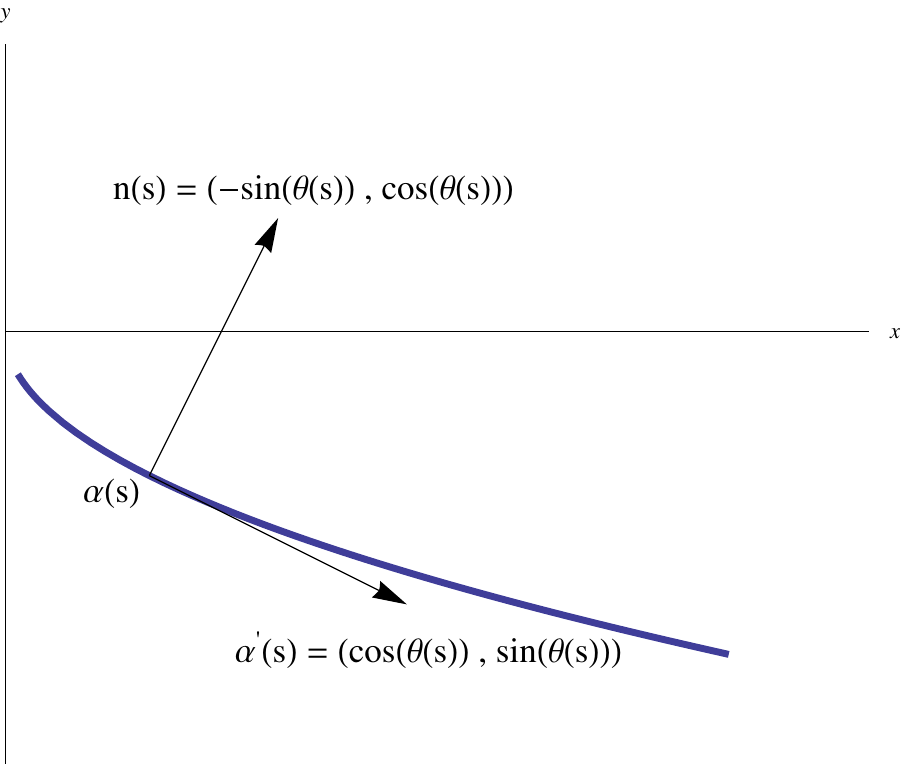}}
\caption{We are assuming that the parameter $s$ is arc-length parameter. The vectors $\alpha^\prime(s)$ and $n(s)$ are perpendicular.}
\label{vectors}
\end{figure}

Let us assume that $\beta(t)=\alpha(vt)$ describes the motion of the block. Since $s=vt$ and $s$ is arc-length parameter, then the  speed of the block is constant, it is $v$. Since we have that $\beta^{\prime\prime}(t)=v^2\alpha(vt)$ then, Figure \ref{fbd} shows the free body diagram for the problem of making the block move along $\alpha$ with constant speed $v$.

\begin{figure}[ht]
\centerline{\includegraphics[width=10.6cm,height=4.2cm]{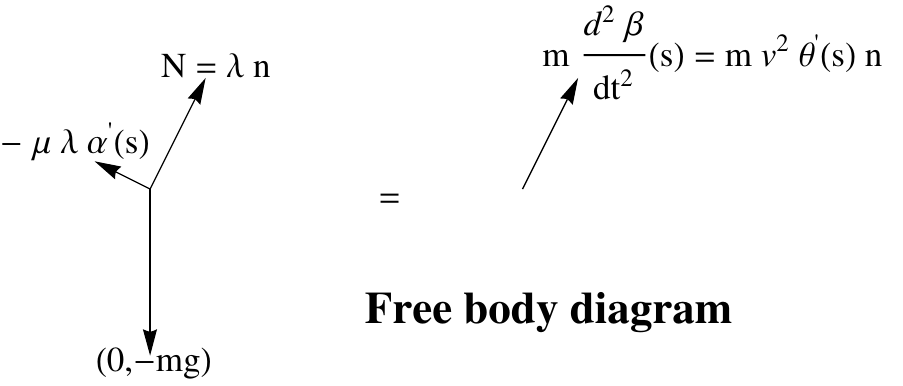}}
\caption{By Newton's second law, the sum of the normal force $N$ plus the weight plus the friction force must be $m\beta^{\prime\prime}$.}
\label{fbd}
\end{figure}

As the free body shows the following equation must hold true.

\begin{eqnarray}\label{nsl}
m v^2 \theta^\prime(s)n(s)=\lambda(s)\, n(s)-\tan(\delta) \lambda(s)\alpha^\prime(s)-(0,mg)
\end{eqnarray}

By doing the inner product of both sides of the equation \ref{nsl} with the vector $\alpha^\prime(s)$ we obtain that 

$$ \lambda(s)=-mg\cot(\delta) \sin(\theta(s))$$

Likewise, by doing the inner product of both sides of the equation \ref{nsl} with the vector $n(s)$, we obtain that $mv^2\theta^\prime=\lambda-mg\cos(\theta)$ and therefore, if $ a=\frac{g}{v^2\sin(\delta)}$, then

\begin{eqnarray}\label{the ode}
\theta^\prime(s)=\frac{-g}{v^2\sin(\delta)}\, \sin(\theta(s)+\delta)=-a\, \sin(\theta(s)+\delta)
\end{eqnarray}

Since the differential equation  (\ref{the ode}) does not have the variable $s$, then, all the solutions differ by a horizontal translation. This is, if $\theta(s)$ is a solution, then $\theta(s+c)$ is also a solution for every real number $c$. Due to the geometry of our problem we do not need to consider an integrating constant, since just one solution will give us all the solutions $\alpha(s)$. Recall that the equilibrium solution of the differential equation (\ref{the ode}) is $\theta(s)=-\delta$ for all $s$. This solution corresponds to the case $\alpha(s)=(\cos(\delta)\, s, -\sin(\delta)\, s)$ which is the straight line ramp shown in Figure \ref{line}. When we solve this differential equation by separation of variables we notice that we need to integrate the function $\csc(\theta+\delta)$. Instead of using the classical formula $\int\csc(u)\, du=-\ln(\csc(u)+\tan(u))$ we will use the formula $\int\csc(u)\, du=\ln(\tan(\frac{u}{2}))$ which led us to the formula

$$\theta(s)=-\delta+2\arctan(\e^{-as})$$

It is clear that if $\gamma(s)=(z(s),w(s))$ denotes a counterclockwise rotation of $\delta$ radians of the curve  $\alpha(s)$, then

$$z^\prime(s)=\cos(2\arctan(\e^{-as}))\com{and} w^\prime(s)=\cos(2\arctan(\e^{-as}))$$

Integrating the equations above we obtain that 

$$z(s)=s+\frac{\ln(1+\e^{-2as})}{a} \com{and} w(s)=\frac{2}{a}\, \hbox{\rm arccot}(\e^{-as})$$

The curve $(z(s),w(s))$ is shown in Figure \ref{alpha1}. As mentioned in the introduction, the solution that we are looking for is a clockwise rotation of the curve $\gamma$ by an angle $\delta$. In the next section we give a more detailed explanation of this solution. 
\section{Understanding the solution of the ODE. A mathematical definition of ramp}\label{deframp}

Let us start this section with a definition.

\begin{mydef}
We will say  that a curve $\gamma:[a_1,a_2]\to \bfR{2}$ is regular is $\gamma^\prime(t)$ never vanishes. We say $n:[a_1,a_2]\to \bfR{2}$ is a normal of the curve $\gamma$ if $n(t)$ has length 1 and the inner product of $n(t)$ and $\gamma^\prime(t)$ is zero, this is,  $n(t)\cdot\gamma(t)=0$ for all $t$. 
\end{mydef}


\begin{mydef}\label{def ramp}
A ramp in the plane $\bfR{2}$ is an ordered pair $(\gamma,n)$ where $\gamma:[a_1,a_2]\to \bfR{2}$ is a  regular curve and $n:[a_1,a_2]\to \bfR{2}$ is a normal. We will interpret  the ramp $(\gamma,n)$ as a portion of the plane  whose boundary contains  $\gamma$ and its outer normal vector along  is $n$. 

\end{mydef}

\begin{example}
$(\gamma_1,n_1)$ where $\gamma_1,\,  n_1:[0,\pi]\to \bfR{2}$ are given by $\gamma_1(t)=5 (\cos(t),\sin(t))$ and $n_1(t)=(\cos(t),\sin(t))$ is an example of a ramp. This ramp is shown in the first two images in Figure \ref{rampex}. $(\gamma_2,n_2)$ where $\gamma_2,\,  n_2:[0,\pi]\to \bfR{2}$ are given by $\gamma_2(t)=5 (\cos(t),\sin(t))$ and $n_2(t)=-(\cos(t),\sin(t))$ is an example of a ramp. This ramp is shown in the last two images in Figure \ref{rampex}.

\end{example}

\begin{figure}[ht]
\centerline{\includegraphics[width=4.6cm,height=3.1cm]{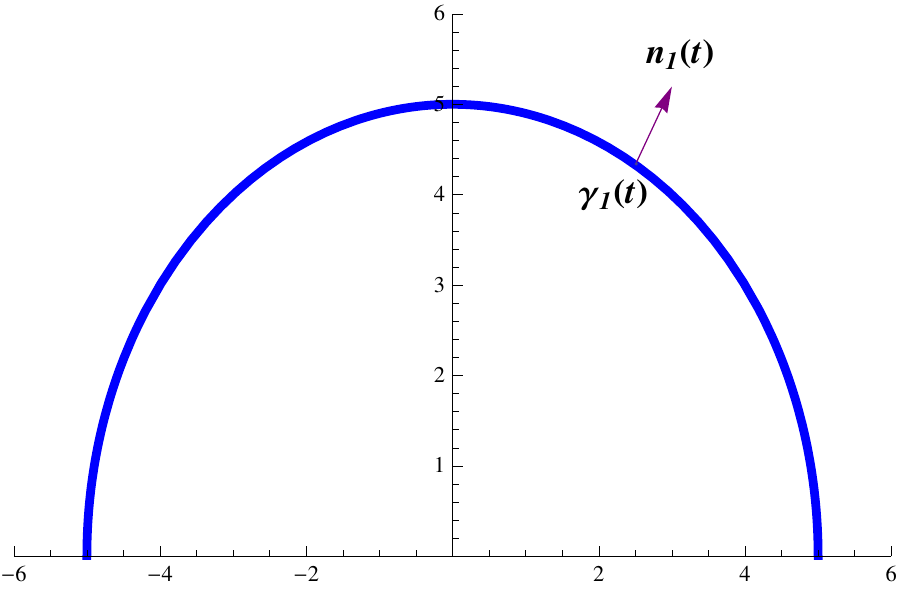}\includegraphics[width=4.6cm,height=3.1cm]{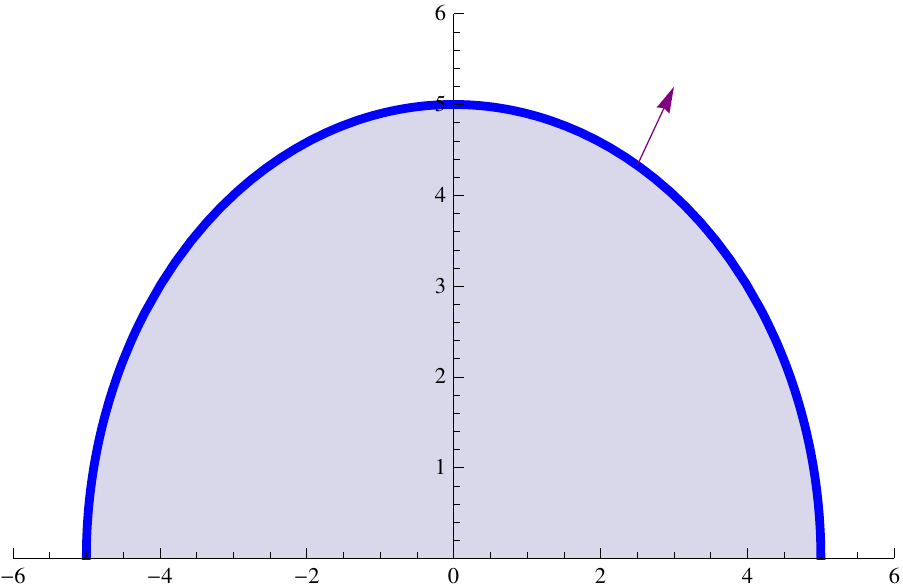}\includegraphics[width=4.6cm,height=3.1cm]{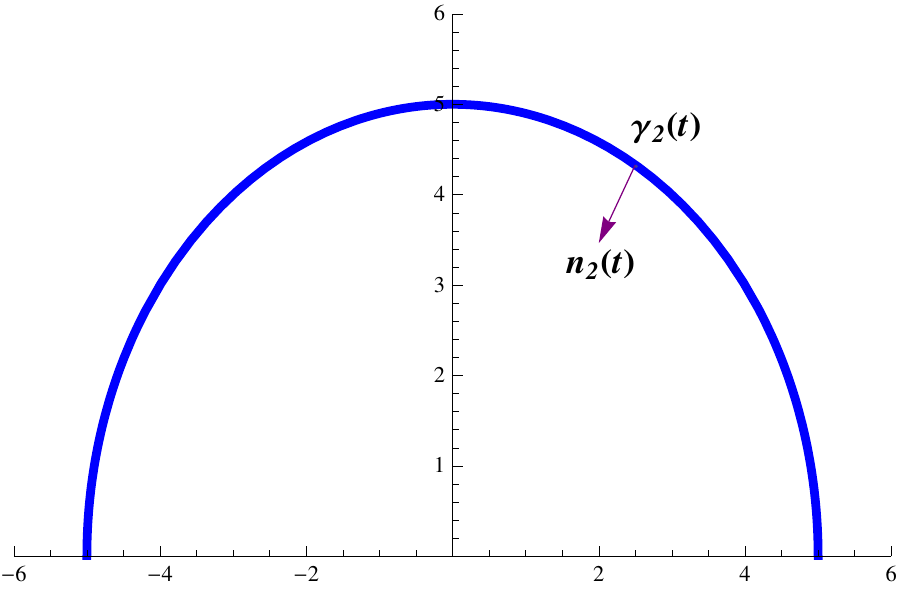}\includegraphics[width=4.6cm,height=3.1cm]{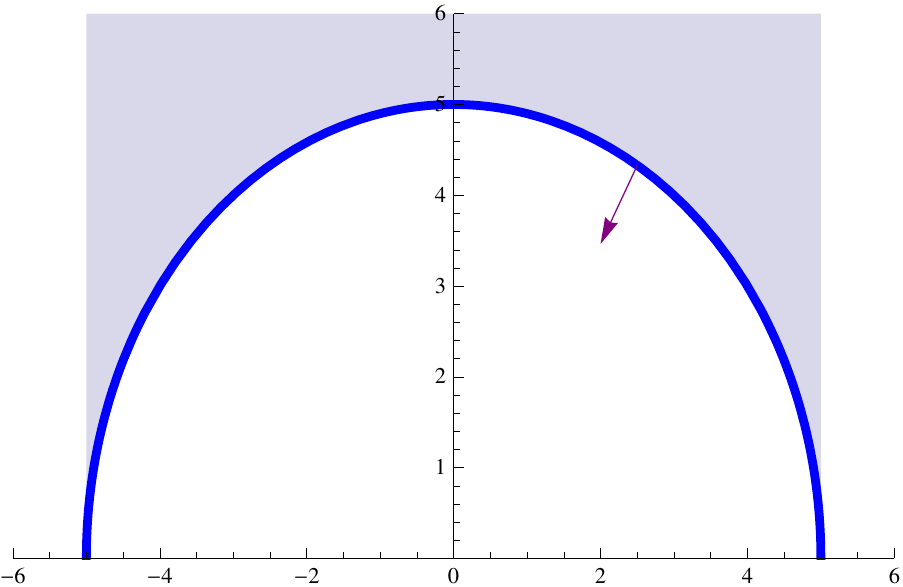}}
\caption{A normal vector of a curve helps us to define the part of a curve where we want a object to slide on a ramp}
\label{rampex}
\end{figure}

The following definition is based on Newton's second law.


\begin{mydef}\label{def sol} A {\bf ramp with external force} is a triple  $(\gamma,n,F)$ where  $F:[a_1,a_2]\to \bfR{2}$ is a smooth function and  $(\gamma,n)$ is a ramp. Given a positive number $m$, a {\bf solution of the ramp with external force $(\gamma,n,F)$} is given by a curve $\beta(t)=\gamma(h(t))$ (Notice that $\beta$ is just a re-pametrization of the curve $\gamma$ and it is completely determined by the function $h:[a_1,a_2]\to {\bf R}$) and a nonnegative function $ \lambda:[a_1,a_2]\to {\bf R} $ such that 

$$F(h(t))+\lambda(t) n(h(t))= m \beta^{\prime\prime}(t)$$

\end{mydef}

\begin{rem}
In Definition \ref{def sol}, the term $\lambda(t) n(t)$ represents the force that the surface of the ramp is doing on the object that moves on the ramp under the action of the external force $F$. By Newton's third law,  $-\lambda(t) n(t)$ is the force that the object is doing to the ramp. The condition $\lambda\ge0$ is needed so that the object stays on the ramp.
\end{rem}

\begin{example}
If  $(\gamma,n)$ is a ramp, $m>0$ and $F(t)=(0,-9.81\, m)$, then the triple  $(\gamma,n,F)$ represents the action of the gravity acting on a particle with mass $m$ that moves on the ramp without friction. 
\end{example}

\begin{example}
Let us consider the ramp $(\gamma,n)$ where $\gamma,\,  n_2:[-1.5,1.5]\to \bfR{2}$ are given by $\gamma(t)=(t,t^2)$ and $n(t)=(\frac{2t}{\sqrt{4t^2+1}},\frac{-1}{\sqrt{4t^2+1}})$.  If $F(t)=(0,-9.81)$ then, the ramp with external force $(\gamma,n,F)$ has no solution.  Figure 
\ref{nsramp} shows this ramp.

\end{example}

\begin{figure}[ht]
\centerline{\includegraphics[width=5cm,height=6cm]{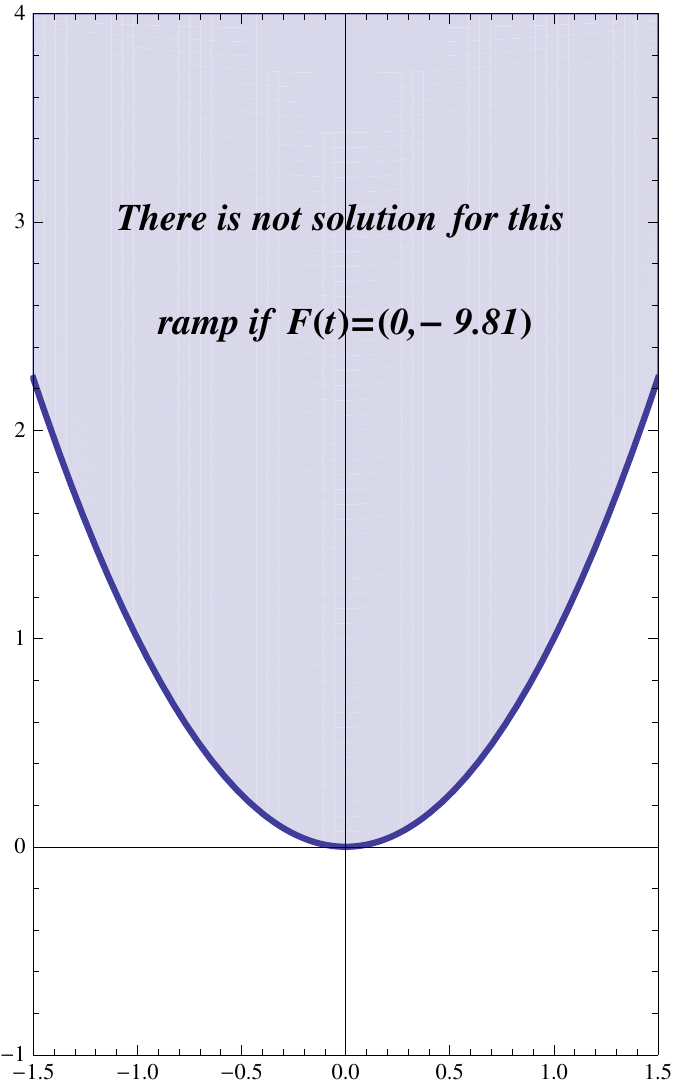}}
\caption{Some ramps with external force do not have a solution.}
\label{nsramp}
\end{figure}


Definition \ref{def sol} provides the definition of the solution of a particle moving on the ramp under the action of the  force $F$.
In the case that $F$ denotes all the forces acting on the particle that moves on the ramp {\it but} the friction force, the definition of solution needs to be changed to:

\begin{mydef}\label{solwfric} Given a ramp with external force  $(\gamma,n,F)$ defined on the interval $[a_1,a_2]$ and two positive numbers $\mu<1$ and $m$, a {\bf solution of the ramp with external force $(\gamma,n,F)$ and kinetic  coefficient of friction $\mu$ } is given by a curve $\beta(t)=\gamma(h(t))$ and a nonnegative function $ \lambda:[a_1,a_2]\to {\bf R} $ such that 

$$F(h(t))+ \lambda(t) n(h(t))-\mu \lambda(t) \frac{\beta^\prime(t)}{|\beta^\prime (t)|}   = m \beta^{\prime\prime}(t)$$

\end{mydef}


\begin{rem}
Let us define an {\it  elementary function} to be a function of one variable with real values built from a finite number of  trigonometric, exponential, constant, 
$n^{th}$ functions  and their inverses, through composition and using the four basic elementary operation (+, -, $\times$, $\div$). When we want to define basic examples of curves defined by arc-length parameter, that is, if we want to define $\gamma(s)=(x(s),y(s))$ such that $x^\prime(s)^2+y^\prime(s)^2=1$ for all $t$, then, the first thing that comes to our mind if to find easy possibilities for $x^\prime(t)$ and $y^\prime(t)$. The easiest one would be 
a pair of  numbers $c_1$  and $c_2$ such that $c_1^2+c_2^2=1$. If we make $x^\prime(s)=c_1$ and $y^\prime(s)=c_2$ then after and easy integration we obtain that the curve $\gamma$ is a straight line. If we want to use the fact that $\cos(as)^2+\sin(as)^2=1$ and we decide to make $x^\prime(s)=\cos(as)$ and $y^\prime(s)=\sin(as)$, then, after an easy integration we obtain that $\gamma$ should be a circle. The most important curves in this paper are those that are obtained by using the identity

$$ \hbox{\rm tanh}(t)^2+\hbox{\rm sech}(t)^2=1\quad \hbox{for all $t$} $$

That is, we will be using curves that satisfy $x^\prime(s)=\hbox{\rm tanh}(as)$ and $y^\prime(s)=\hbox{\rm sech}(as)$.

\end{rem}

The following lemma is a direct computation and provides  a definition of the curve $\alpha$ whose graph is shown in Figure \ref{alpha1}

\begin{lem}\label{lemma1}
For any non zero real number $a$, the curve 

\begin{eqnarray}\label{alpha}
\alpha(s)=(x(s),y(s)))=\left( s+\frac{1}{a}  \ln(1+\e^{-2as}), \frac{2}{a}\, \hbox{\rm arccot}(\e^{-as}) \right)
\end{eqnarray}

is an arc-length parametrized curve. Moreover,

$$x^\prime(s)=\hbox{\rm tanh(as)}\com{and} y^\prime(s)=\hbox{\rm sech(as)}$$

\end{lem}


\begin{lem}\label{lemma2}
For any  positive number $\delta$ smaller than $\frac{\pi}{2}$, if $\alpha_\delta$ represents the curve obtained by  rotating  an  angle of $\delta$ radians the curve $\alpha=(x(s),y(s))$ defined in Lemma \ref{lemma1}, this is, if 
\begin{eqnarray}\label{alphabeta}
\alpha_\delta(s)=(x_\delta(s),y_\delta(s))=\left(\cos(\delta)x(s)+\sin(\delta) y(s),-\sin(\delta) x(s)+\cos(\delta)y(s) \right)
\end{eqnarray}

then, the maximum value for the second entry  of $\alpha_\delta$ is achieved when $s=\frac{\hbox{\rm arcsinh($\cot(\delta)$)}}{a}$. Also, we have that 

\begin{eqnarray}\label{normal}
\alpha^{\prime\prime}_\delta(s)\cdot (y^\prime_\delta(s),-x^\prime_\delta(s))= a\, \hbox{\rm sech}
(as)\end{eqnarray}

The graph of the curve $\alpha_\delta$ is shown in Figure \ref{alpha23}.

\end{lem}

\begin{proof}
A direct computation using Lemma \ref{lemma1} shows that $y_\delta^\prime(s)=-\sin(\delta)\hbox{\rm tanh}(as)+\cos(\delta)\hbox{\rm sech}(as)$, then we can see that the only solution of the equation $y_\delta^\prime(s)=0$ is $s=\frac{ \hbox{\rm arcsinh($\cot(\delta)$)}}{a}$. In order to prove the identity \ref{normal} we point out that since the inner product is invariant under rotations, this identity is equivalent to show that $\alpha^{\prime\prime}(s)\cdot (y^\prime(s),-x^\prime(s))= a\, \hbox{\rm sech}(as)$ which follows because,

$$ x^{\prime\prime}(s)y^\prime(s)-y^{\prime\prime}(s)x^\prime(s)=a\hbox{\rm sech}^3(as)+a \hbox{\rm tanh}^2(as)\hbox{\rm sech}(as)=a\hbox{\rm sech}(as)\, (\hbox{\rm sech}^2(as)+\hbox{\rm tanh}^2(as))=a\hbox{\rm sech}(as)$$
\end{proof}

\begin{rem}\label{dydelta}
In the previous proof we got that $y_\delta^\prime(s)=-\sin(\delta) \hbox{\rm sech}(as) (\hbox{\rm sinh}(as)-\cot(\delta))$. It follows that $y_\delta^\prime(s)>0$ if $s<s_0$ and $y_\delta^\prime(s)<0$ if $s>s_0$.
\end{rem}


\begin{mydef} \label{ramp def}
For any  positive number $\delta$ smaller than $\frac{\pi}{2}$ and any $a>0$ we define the ramps $(\gamma_\delta,n_\delta)$ and $(\tilde{\gamma}_\delta,\tilde{n}_\delta)$ by 

$$\gamma_\delta(s)=\alpha_\delta(s_0-s) \com{and} n_\delta(s) = (y_\delta^\prime(s_0-s),-x_\delta^\prime(s_0-s)) $$

and

$$\tilde{\gamma}_\delta(s) = \alpha_\delta(s+s_0) \com{and} \tilde{n}_\delta(s) = (-y_\delta^\prime(s+s_0),x_\delta^\prime(s+s_0)) $$

where $s_0=\frac{\hbox{\rm arcsinh($\cot(\delta)$)}}{a}$ and the maps $\gamma_\delta$ and $n_\delta$ are defined in Lemma \ref{lemma2} and all the function $\alpha_\delta, n_\delta, \tilde{\alpha}_\delta, \tilde{n}_\delta$ are defined on the interval  $[0,\infty)$. These two ramps are shown in Figure \ref{fig ramps}.

\end{mydef}


\begin{figure}[ht]
\centerline{\includegraphics[width=6.1cm,height=4.6cm]{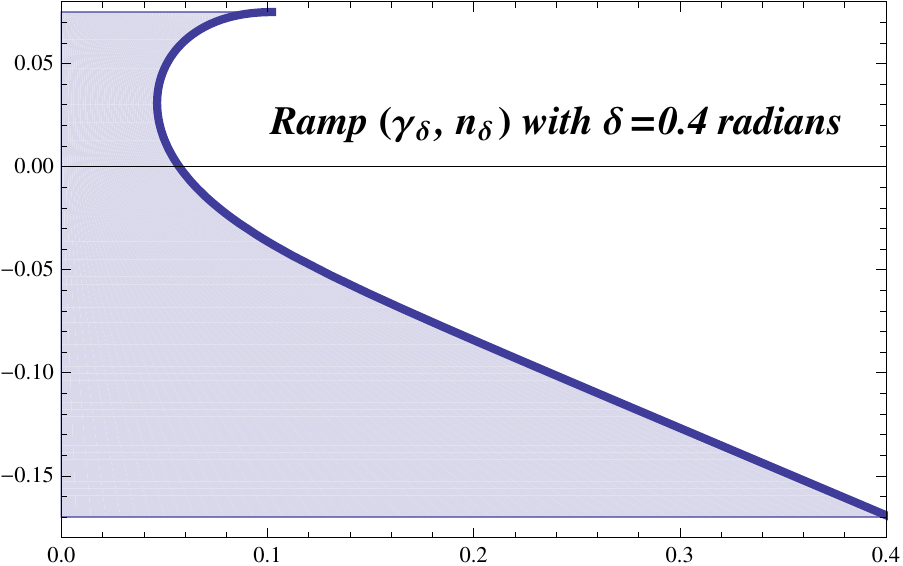}\hskip1cm\includegraphics[width=6.6cm,height=4.1cm]{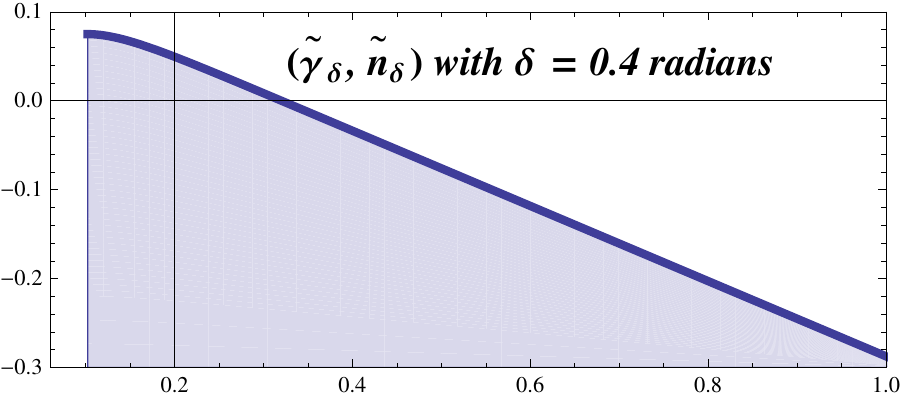}}
\caption{The two ramps define in Definition \ref{ramp def}}
\label{fig ramps}
\end{figure}

Let us state and proof the main theorem in this section.


\begin{thm}  Given an angle $\delta$ between $0$ and $\frac{\pi}{4}$ radians and a positive speed $v>0$. Let us consider $\mu=\tan(\delta)$, $a=\frac{g}{v^2 \sin\delta}$ and $F(t)=(0,-mg)$ with $g=9.81$. We have

{\bf a.)} For the  ramp $(\gamma_\delta,n_\delta)$ given in Definition \ref{ramp def},

$$ \beta(t)= \gamma_\delta (vt) \com{and} \lambda(t)= m g \cot(\delta) \, y_\delta^\prime (s_0-vt)$$

is a solution for the  ramp with external force $(\gamma_\delta,n_\delta,F)$ and kinetic coefficient of friction $\tan(\delta)$. 

{\bf b.)}  For the ramp   $(\tilde{\gamma}_\delta,\tilde{n}_\delta)$    given in Definition \ref{ramp def},

$$ \beta(t)=\tilde{ \gamma}_\delta (vt) \com{and} \lambda(t)=-m g \cot(\delta) \, y_\delta^\prime (s_0+vt)$$

is a solution for the  ramp with external force $(\tilde{\gamma}_\delta,\tilde{n}_\delta,F)$  and kinetic coefficient of friction $\tan(\delta)$. 

{\bf c.)} Since in both cases $|\beta^\prime(t)|=v$, then these motions have constant speed.

\end{thm}

\begin{proof}

Let us prove part {\bf a.)}. First of all, notice that as it is required by Definition \ref{solwfric}, $\lambda>0$ by remark \ref{dydelta}. From the definition of $\beta$   we obtain that 

$$ \beta^\prime(t)=-v\alpha^\prime_\delta(s_0-vt) \com{and} \beta^{\prime\prime}(t)=v^2\, \alpha_\delta^{\prime\prime}(s_0-vt)$$

therefore the equation

$$F(h(t))+ \lambda(t) n(h(t))-\mu \lambda(t) \frac{\beta^\prime(t)}{|\beta^\prime (t)|}   = m \beta^{\prime\prime}(t)$$

is equivalent to 

\begin{eqnarray}\label{eq1}
(0,-mg)+ \lambda(t)n_\delta(vt) + \tan(\delta) \lambda(t) \alpha_\delta^\prime(s_0-vt)   = m v^2 \alpha_\delta^{\prime\prime}(t)
\end{eqnarray}

Notice that the vector ${\bf u_1}=n_\delta(vt)$ and ${\bf u_2}= \alpha_\delta^\prime(s_0-tv)$ form an orthonormal basis. Therefore in order to prove equation (\ref{eq1}) it is enough to prove that the dot product with ${\bf u_1}$ and   ${\bf u_2}$ of the left hand side (LHS) and right hand side (RHS) of the equation is the same. The dot product of the LHS of equation (\ref{eq1}) with ${\bf u_1}$  is equal to 

\begin{eqnarray*}
mg\, x_\delta^\prime(s_0-vt) +\lambda(t)&=&mg\, x_\delta^\prime(s_0-vt) +\cot(\delta) y^\prime_\delta(s_0-vt)\\
  &=&  mg\left(\cos(\delta) \hbox{\rm tanh}(a(s_0-vt))+\sin(\delta) \hbox{\rm sech}(a(s_0-vt))   \right)+\\
 & &m g\cot(\delta) \left(-\sin(\delta) \hbox{\rm tanh}(a(s_0-vt))+
   \cos(\delta) \hbox{\rm sech}(a(s_0-vt))   \right)\\
   &=& \frac{mg}{\sin(\delta)} \,   \hbox{\rm sech}(a(s_0-vt)) 
\end{eqnarray*}

Using Equation (\ref{normal}) we get that the dot product of the RHS  of equation (\ref{eq1}) with ${\bf u_1}$  is equal to 

$$ mv^2 a \hbox{\rm sech} (a (s_0-vt))= \frac{mg}{\sin(\delta)} \,   \hbox{\rm sech}(a(s_0-vt))$$

Therefore LHS $\cdot$ $ {\bf u_1}$ =RHS $\cdot $ $ {\bf u_1}$. It is easy to check that the dot product of the RHS of equation (\ref{eq1}) with ${\bf u_2}$ vanishes. On the other hand, the dot product of the LHS of the equation (\ref{eq1}) with ${\bf u_2}$ equals to

$$ -mg\, y_\delta^\prime(s_0-vt)+v\tan(\delta)\, \lambda=-mg\, y_\delta^\prime(s_0-vt)+\tan(\delta)\, mg\cot(\delta) y_\delta^\prime(s_0-vt)
=0$$ 

Therefore part {\bf a.)} follows. Part {\bf b.)} is similar. Part {\bf c.)} follows because $|\alpha^\prime|=1$ and since $\alpha_\delta$ is a rotation of $\alpha$ then $|\alpha^\prime_\delta|=1$, then, $\beta^\prime(t)=-v\alpha^\prime(s_0-vt)$ and $|\beta^\prime(t)|=v$
 
\end{proof}
\section{3-Dimensional ramps}\label{3dramps}

In this section we describe ramps in the space on which an object can move with constant speed $v_0$ under the assumption that the kinetic coefficient of friction is $\mu$. We will prove that fixing $v_0$ and $\mu$, there are as many different ramp as continuous unit tangent vector fields in the south hemisphere. In the correspondence that we will establish, the ramp that we defined in section \ref{ramp2d} corresponds to a particular choice of the a tangent vector field. 

For curves that represent planar ramps, we got a description of them by studying first their velocity vector. Recall that the function $\theta(s)$ was used to described a point in the lower part of the semicircle that represented a possible velocity vector of the curve that describe the ramp. See Figure \ref{lhs}. When this curve is free to move in the whole 3-dimensional space, the tangent unit vector to the curve lies in the unit sphere and since the block must go down the ramp we can assume that the tangent unit vector of the ramp lies in the shouth hemisphere. See Figure \ref{lhs}. 

\begin{figure}[ht]
\centerline{\includegraphics[width=4.9cm,height=4.cm]{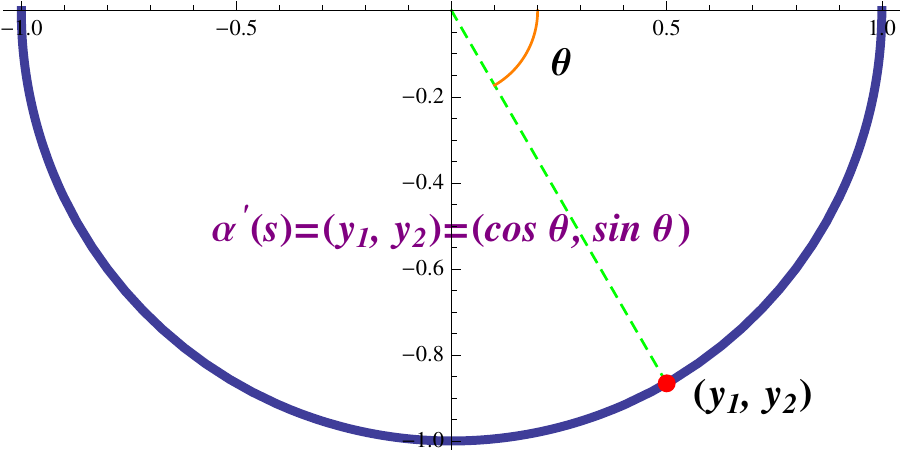}\includegraphics[width=4.6cm,height=4.6cm]{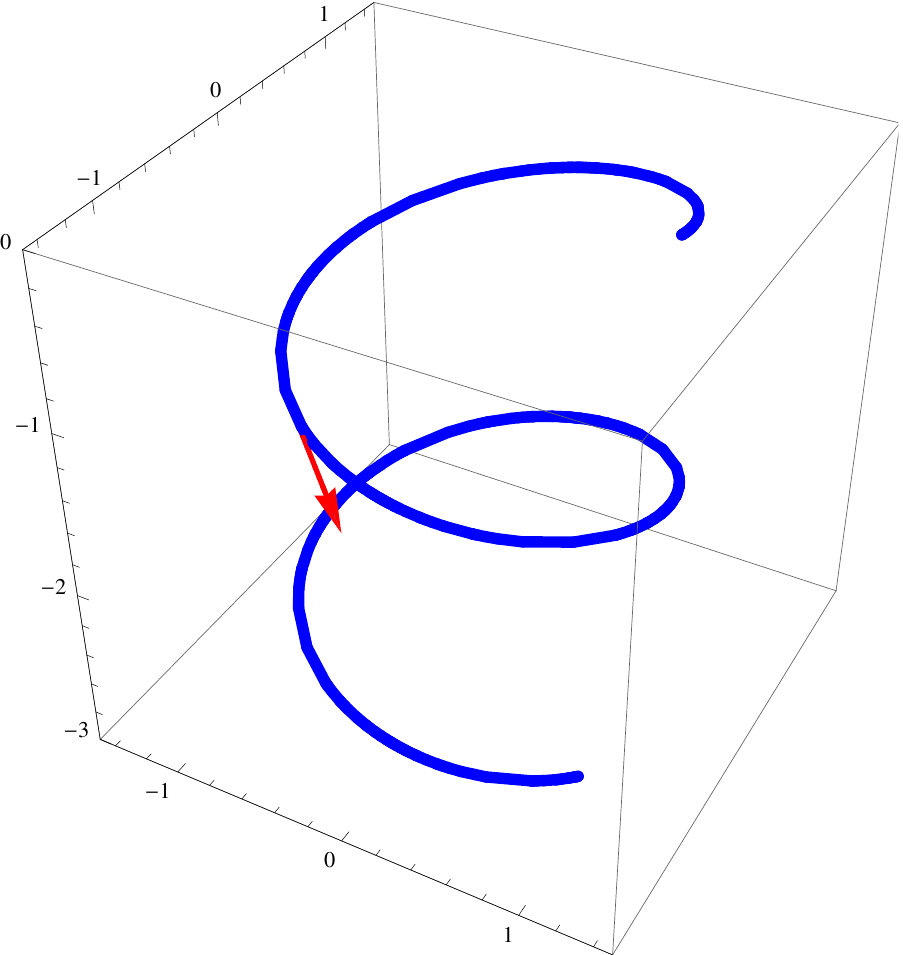}\includegraphics[width=4.6cm,height=4.6cm]{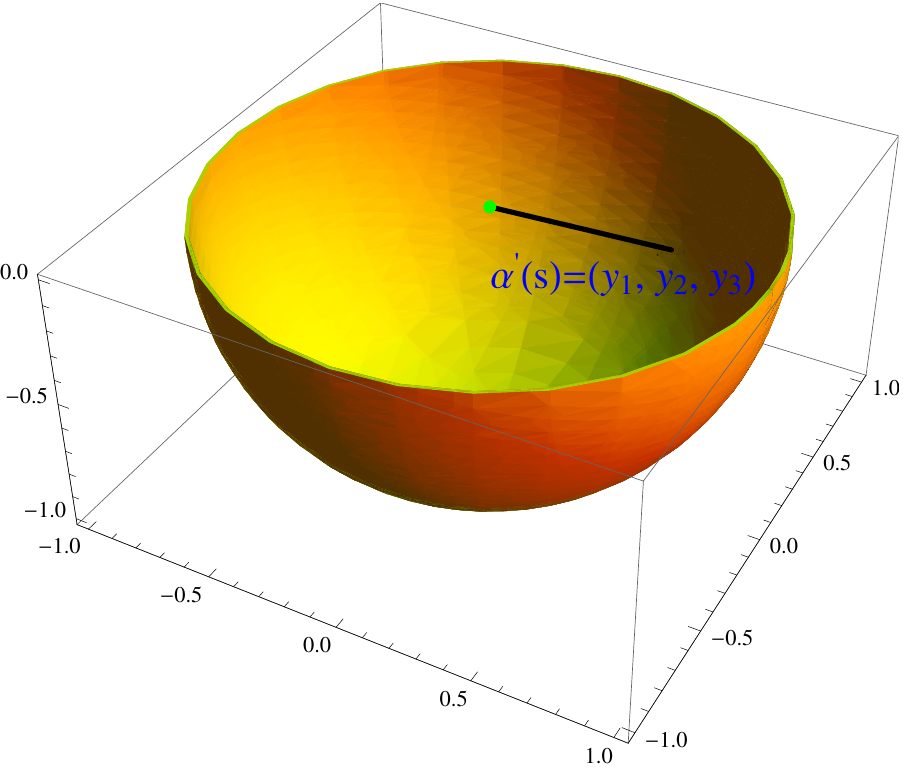}}
\caption{The unit tangent vector of a possible ramp in the plane lies in the lower part of the unit semicircle, if the ramp is free to move in the whole space, then the unit tangent vector lies in the lower hemisphere of the unit sphere}
\label{lhs}
\end{figure}

in order to completely determine the ramp we need to specify the desire direction of the normal to the ramp, see Figure \ref{planes}.  We can do this by choosing a unit normal vector field $H(y)$ defined in the lower hemisphere and requesting the normal vector of the ramp to be equal to $N(y)$ anytime the tangent unit vector to the ramp is the vector $y$. See Figure \ref{normalons}.

\begin{figure}[ht]
\centerline{\includegraphics[width=4.9cm,height=4.cm]{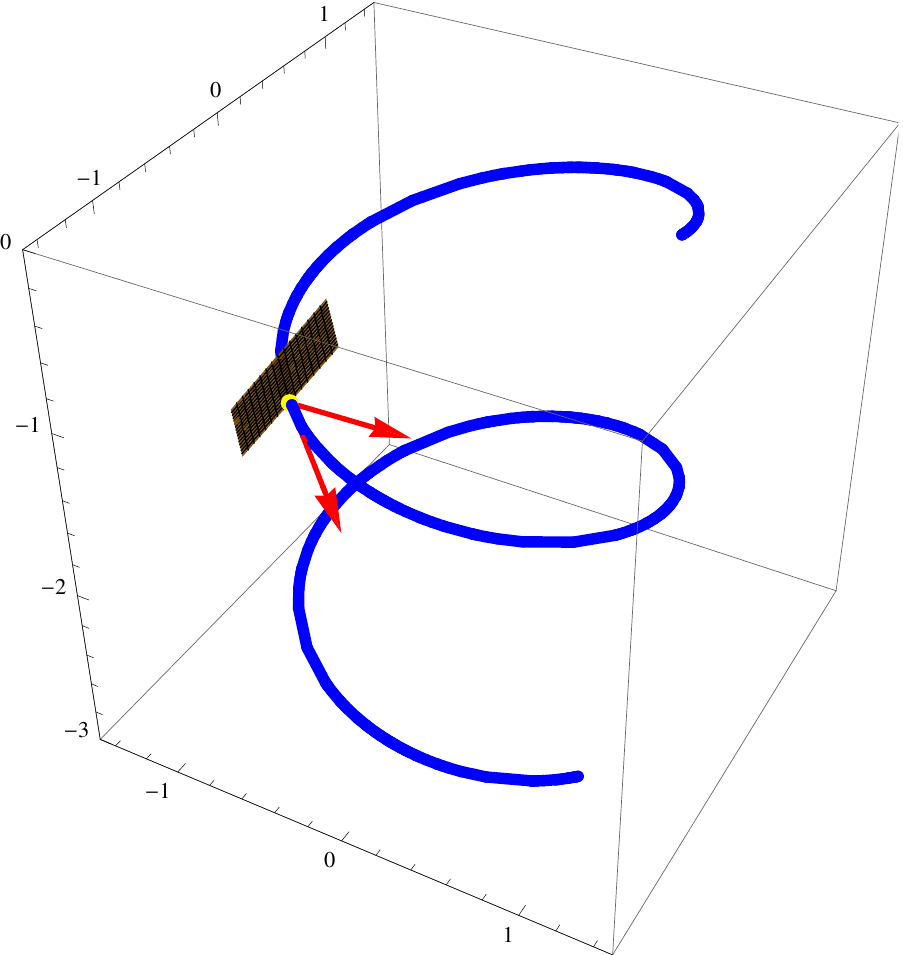}\includegraphics[width=4.6cm,height=4.6cm]{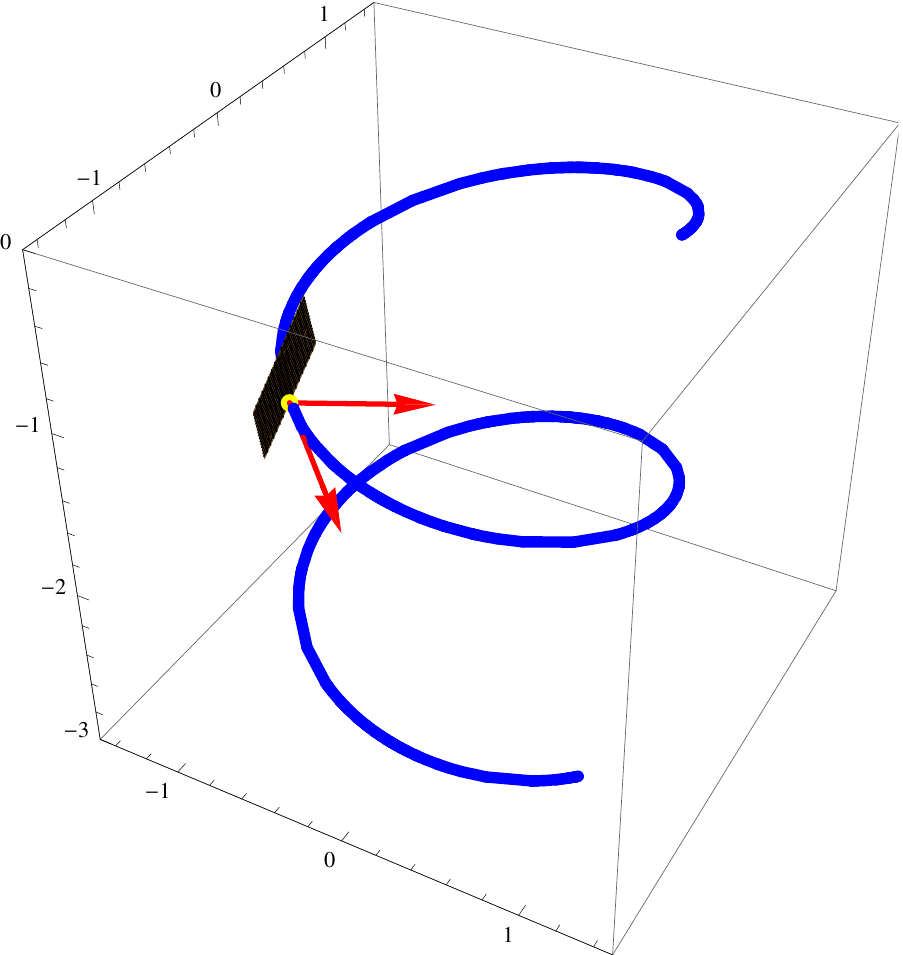}\includegraphics[width=4.6cm,height=4.6cm]{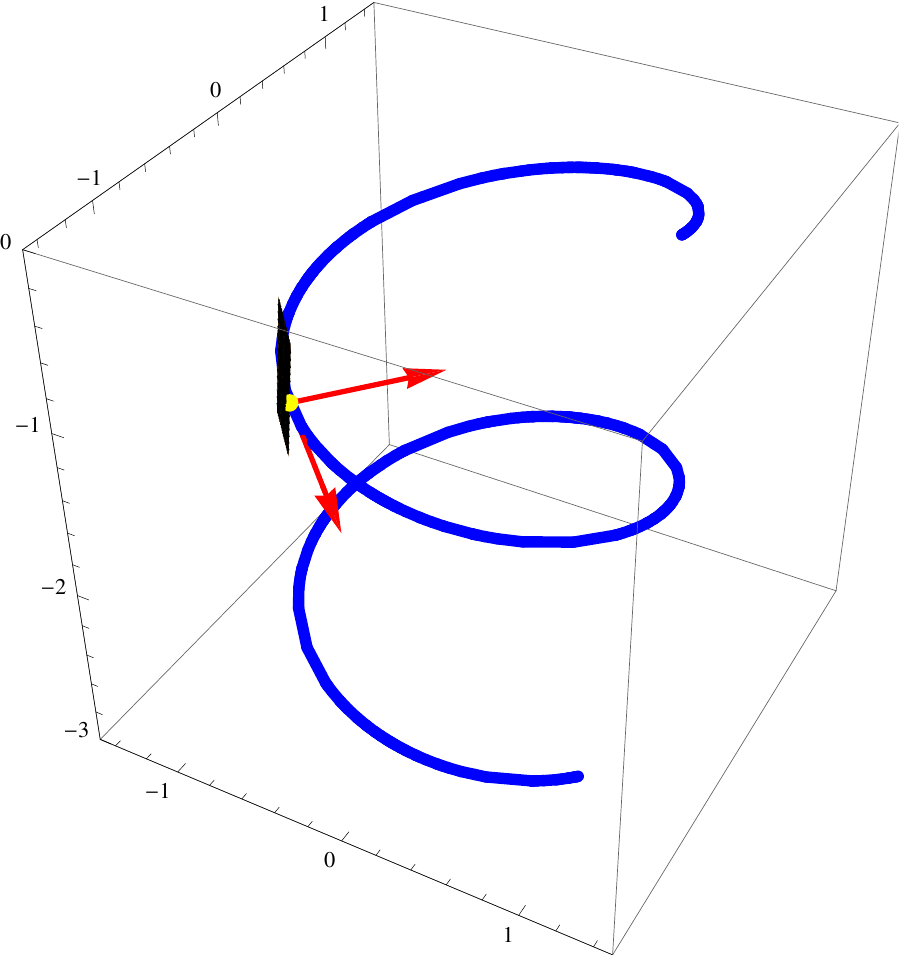}}
\caption{To built a ramp we need to decide about the normal direction at every point in the ramp. }
\label{planes}
\end{figure}

\begin{figure}[ht]
\centerline{\includegraphics[width=4.9cm,height=4.cm]{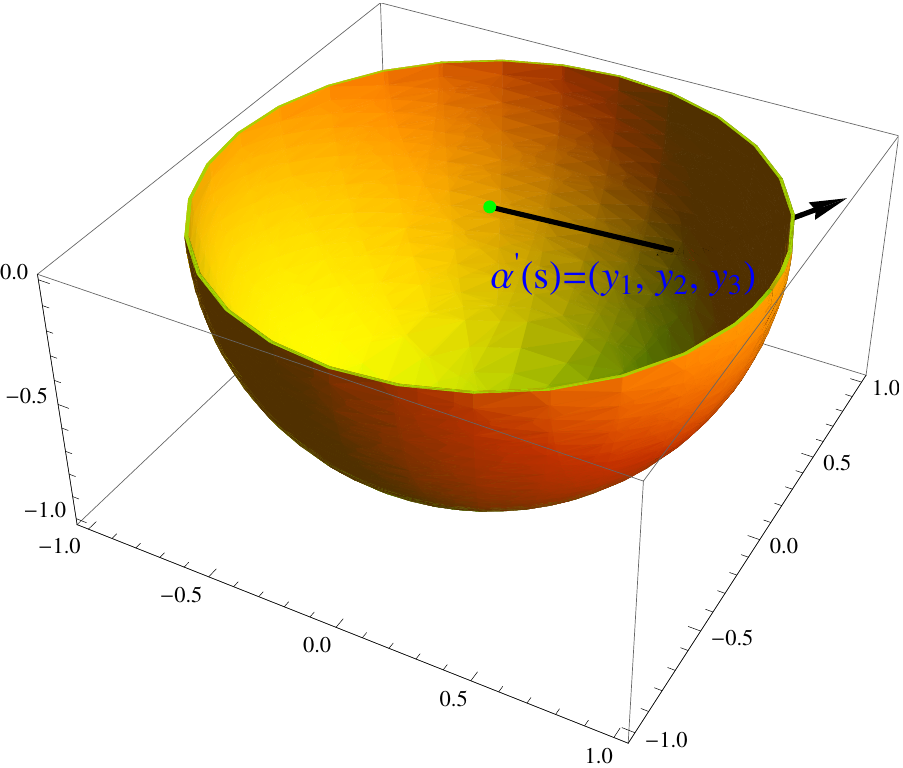}}
\caption{We will make the choice of the normal direction on the ramp to depend on the unit tangent vector.}
\label{normalons}
\end{figure}

 We will prove that for any continuous choice of a normal field $N(s)$ in the lower hemisphere, there exist a family of ramp on which a motion with constant speed is possible under the effect of the gravity force and the force of friction. More precisely we have.
 \begin{thm} Let $\Sigma=\{(y_1,y_2,y_3)\in \bfR{3}:y_1^2+y_2^2+y_3^2=1\, \hbox{and}\quad y_3\le0\}$ denote the south hemisphere and let $N:\Sigma\to \bfR{3}$ be any unit tangent vector field. This is, for any $y\in \Sigma$, $N(y)$ has norm 1 and  $N(y)$ is perpendicular to $y$. For any positive numbers  $v$, $\mu<1$ and any $y_0\in \Sigma$ there exists a unique curve $\alpha(s)$ parametrize by arc-length, such that $\alpha(0)=(0,0,0)$, $\alpha^\prime(0)=y_0$ and the motion along $\beta(t)=\alpha(vt)$ in the ramp given by 
 
 $$R(s,r)= \alpha(s)+r \, \, \alpha^\prime(s)\times N(\alpha^\prime(s))$$
 
 represents a motion that is a solution of Newton's second law under the assumption that the only forces action on this moving particle are gravity and a friction force with kinetic constant coefficient $\mu$. Recall that since the parameter $s$ is arc-length parameter,  then $|\beta^\prime(t)|=v$ for all $t$; this is, the motion has constant speed $v$.
 \end{thm}

\begin{proof}
Recall that we are denoting by $u\cdot w$ the dot (or inner) product of the two vectors $u$ and $w$. A direct computation shows that if  $e_3=(0,0,1)$, then $e_3^T:\Sigma\to \bfR{3}$ given by $e_3^T(y)=e_3-(e_3\cdot y) \, y$ is a tangent vector field (notice that $e_3^T(y)$ must be in the tangent space $T_y\Sigma$ because $e_3^T(y)\cdot y=0$)

Let $\lambda:\Sigma\to{\bf R}$ be the function given by 

$$\lambda(y)=-\frac{g}{ \mu} e_3\cdot y=-\frac{g}{ \mu} y_3$$

 Notice that since $y_3\le0$ then $\lambda\ge 0$ and $\lambda$ only vanishes on the boundary of $\Sigma$, (recall that the boundary of $\Sigma$ is given by the equation $y_3=0$). Let us consider the tangent vector field 

\begin{eqnarray}\label{vector field}
X=-\frac{1}{v^2} \left( g e_3^T - \lambda N(y)\right)=-\frac{g}{v^2} \left( e_3^T + \frac{y_3}{\mu} N(y)\right)
\end{eqnarray}

Notice that along points in the boundary of $\Sigma$, $X(y)=-\frac{g}{v^2}\, e_3$ points toward  $\Sigma$. It follows that any integral curve of the tangent vector field $X$ remain in $\Sigma$.  Let  $\gamma(s)$ be the integral curve of the vector field $X$ that satisfies $\gamma(0)=y_0$.  We will prove the Theorem by showing the curve  $\alpha(s)$ given by

$$\alpha(s)=\int_0^s\, \gamma(u)\, du$$

satisfies all the conditions. Clearly, $\alpha(0)=(0,0,0)$ and $\alpha^\prime(0)=\gamma(0)=y_0$. We also have that $s$ is arc-length parameter because $|\alpha^\prime(s)|=|\gamma(s)|=1$. A direct computation show that the normal vector of the surface (or ramp) $R$ along the curve $\alpha$ (along points in the surface with $r=0$) is given by $N(\alpha(s))$. We will now proof that Newton's second law holds true for $\beta(t)=\alpha(vt)$. Notice that since $v$ is constant, then $\beta^\prime(t)=v\alpha^\prime(vt) =v\gamma(s)$ and $\beta^{\prime\prime}(vt)=v^2\gamma^\prime(vt)=v^2\gamma^\prime(s)$. If $m$ denotes the mass of the particle and $\lambda$ is defined as in the beginning of the proof, we have that 

\begin{eqnarray*}
m\beta^{\prime\prime}(t)=m v^2 \gamma^\prime(s)&=&-mv^2\frac{1}{v^2} \left( g e_3^T(\gamma)-\lambda N(\gamma)\right)\\
&= &m(-g e_3+ g (e_3\cdot \gamma)\, \gamma -  \frac{g}{\mu} (e_3\cdot \gamma)\, N(\gamma(t)))\\
&=& -gme_3+m \lambda N(\gamma(s))- m \lambda \mu \gamma(s)
\end{eqnarray*}

From the equation above we conclude that then normal force made from the ramp to the block has magnitude $m\lambda$. Notice that the last equation above is Newton's second law. 

It is clear the the form of the ramp depends on the desire constant speed that we want on the ramp. For example if the block is to travel outside down on some portion of the ramp and the speed $v$ is not much then the curvature of that portion of the ramp must be big so that the block does not fall off the ramp. The following corollary explains how to change the ramp if we want to change the constant speed or if we want to change the gravity force.
\end{proof}
 
\begin{cor}
{\bf a.)} If a ramp $R\subset \bfR{3}$ does the work of allowing a block to move down with constant speed $v$ and  gravity $g$, then, for any positive $\kappa$ the ramp $\kappa R=\{\kappa (x,y,z): (x,y,z)\in R\}$ does the work of allowing a block to move down with constant speed $\sqrt{\kappa} v$ and gravity $g$.

{\bf b.)} If a ramp $R\subset \bfR{3}$ does the work of allowing a block to move down with constant speed $v$ and gravity $g$, then, for any positive $\kappa$ the ramp $\kappa R=\{\kappa (x,y,z): (x,y,z)\in R\}$ does the work of allowing a block to move down with constant speed 
$v$ and gravity $\frac{g}{\kappa}$.

\end{cor}

\begin{proof}
This corollary is a consequence of Equation (\ref{vector field}). We can check that if $\gamma(t)$ is an integral curve of the vector field $X$, and $\alpha(s)=\int_0^s\, \gamma(u)\, du$, then $\tilde{\gamma}(\tau)=\gamma(\frac{\tau}{\kappa })$ is a solution of the vector field $\frac{1}{\kappa} \, X$, which can be interpreted as either the vector field coming from the tangent unit vector field $N(y)$ with velocity $\sqrt{\kappa}\, v$ and gravity $g$ or it can be interpreted as the vector field coming from the tangent unit vector field $N(y)$, with velocity $v$ and gravity $\frac{g}{\kappa }$. The result follows by noticing that 

$$\tilde{\alpha}(s)=\int_0^s\, \tilde{\gamma}(u)\, du=\int_0^s\, \gamma(\frac{u}{\kappa})\, du=
\kappa\, \int_0^\frac{s}{\kappa}\, \gamma(v)\, dv=\kappa\, \alpha(\frac{s}{\kappa})$$

\end{proof}

\begin{rem}
If a ramp $R$ on Earth has the property than an object will fall down with constant speed $v$, then, if we dilate this ramp by a factor of $6$, the same block will move down with the same constant speed $v$ on this dilated ramp when it is placed on the Moon.
\end{rem}

\vfil
\eject

\begin{figure}[ht]
\centerline{\includegraphics[width=4.6cm,height=4.6cm]{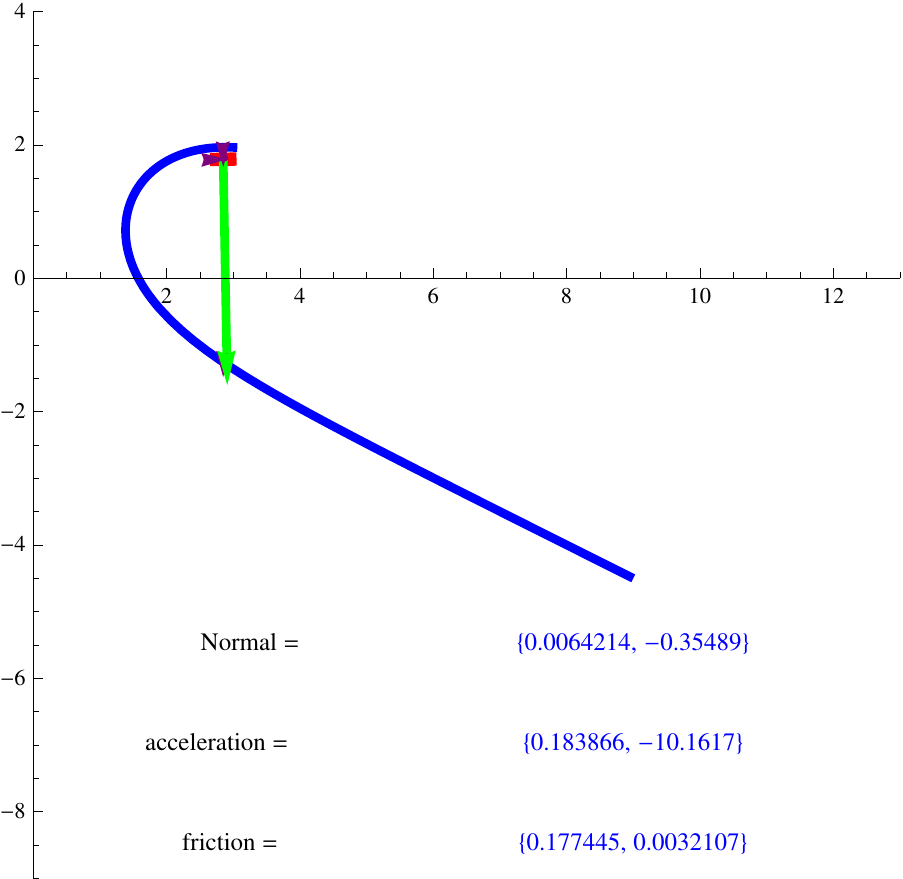}\includegraphics[width=4.6cm,height=4.6cm]{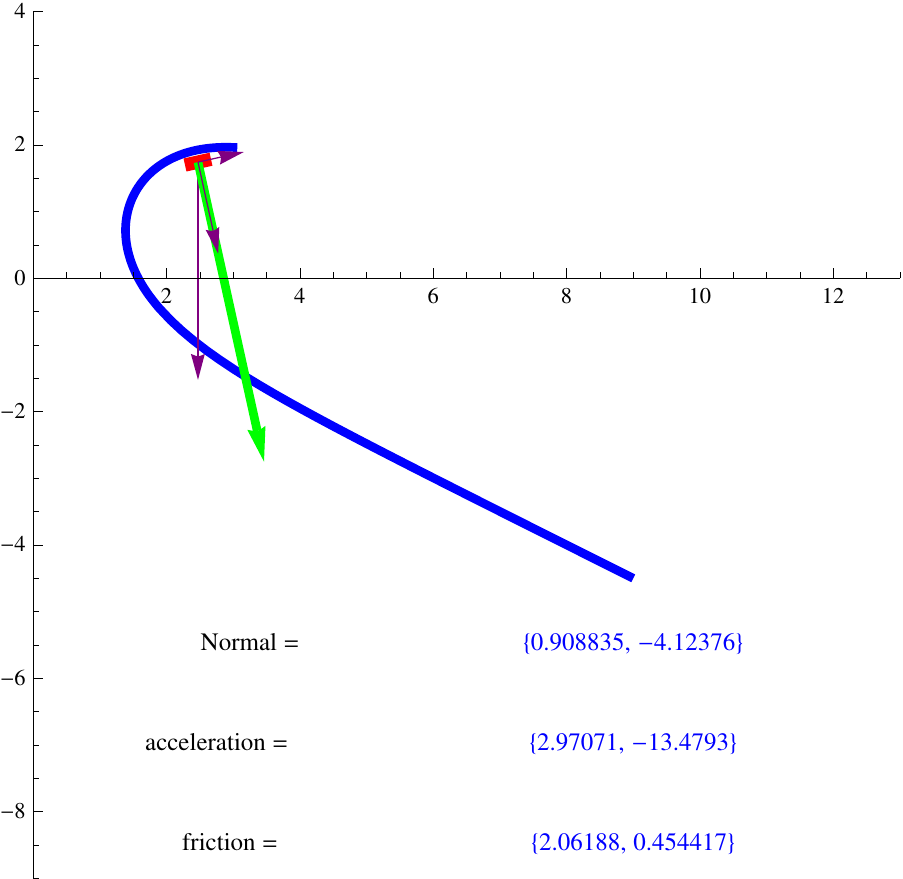}\includegraphics[width=4.6cm,height=4.6cm]{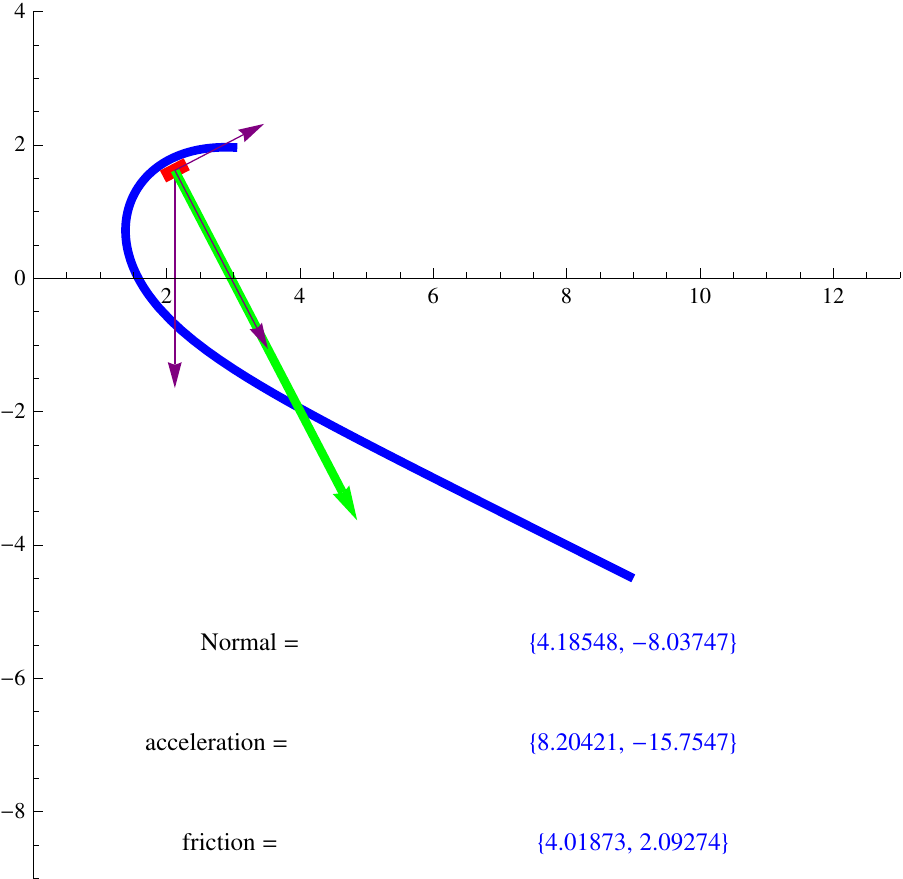}\includegraphics[width=4.6cm,height=4.6cm]{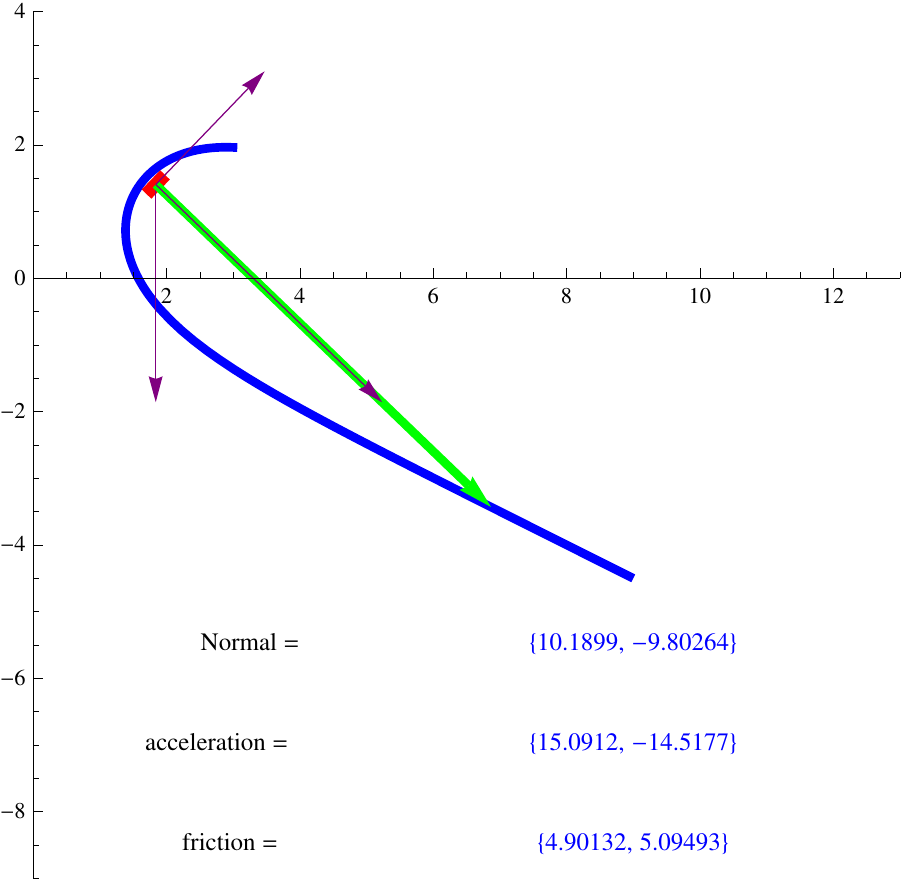}}
\end{figure}
\begin{figure}[ht]
\centerline{\includegraphics[width=4.6cm,height=4.6cm]{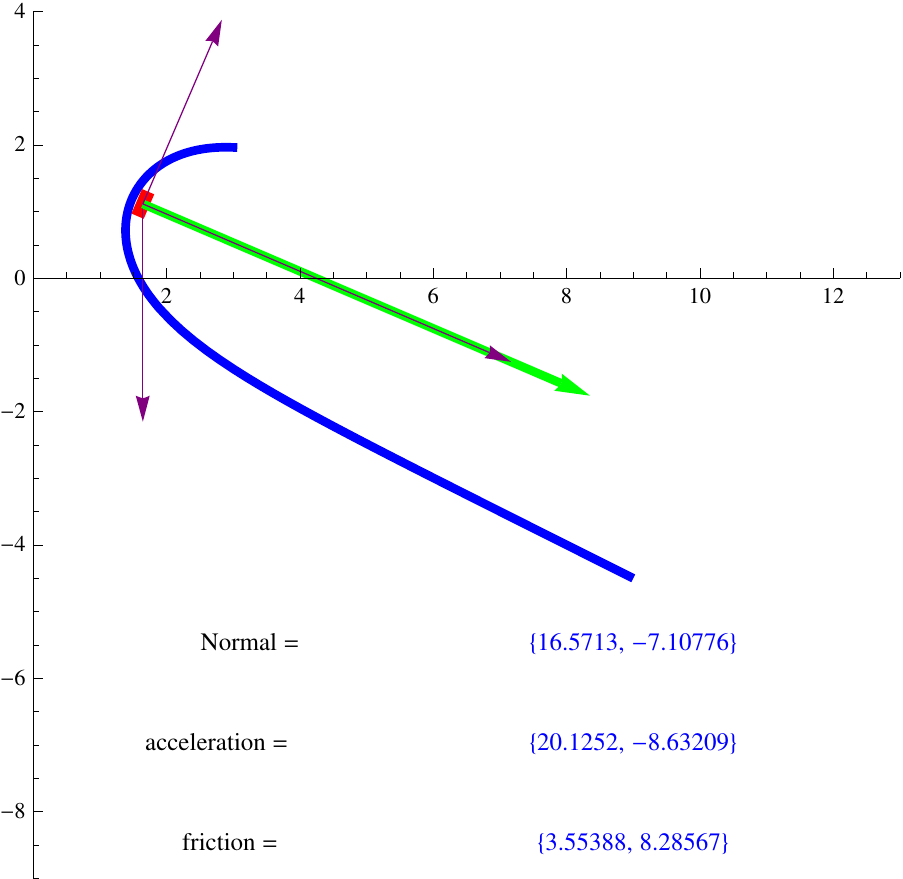}\includegraphics[width=4.6cm,height=4.6cm]{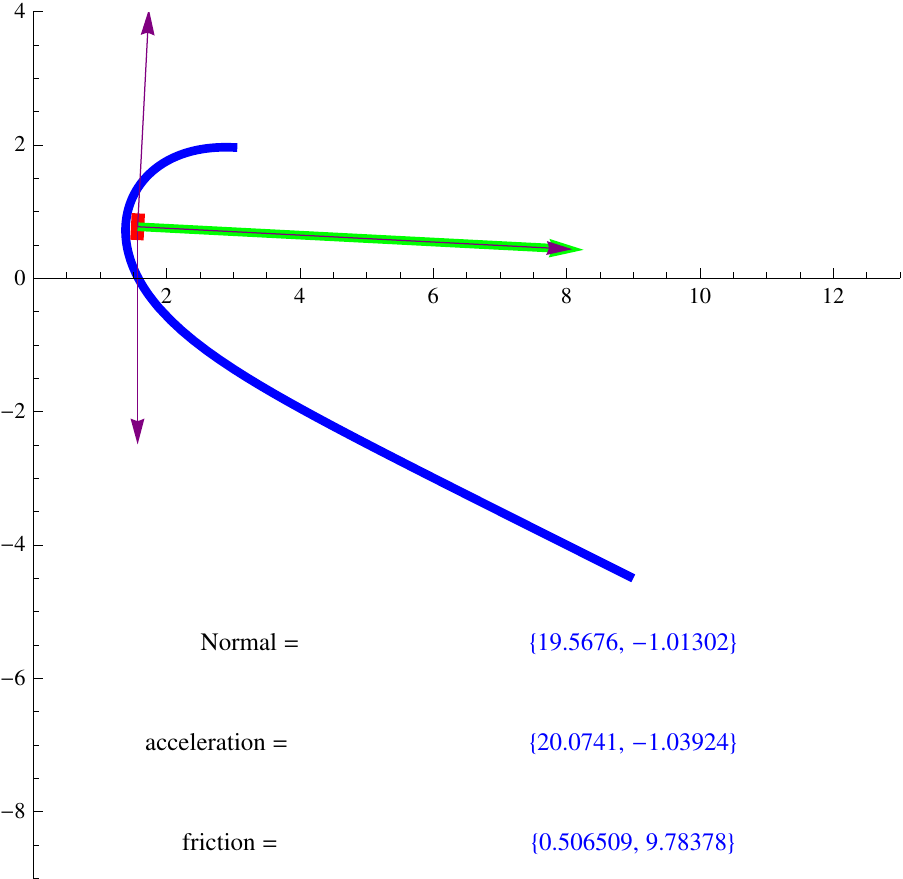}\includegraphics[width=4.6cm,height=4.6cm]{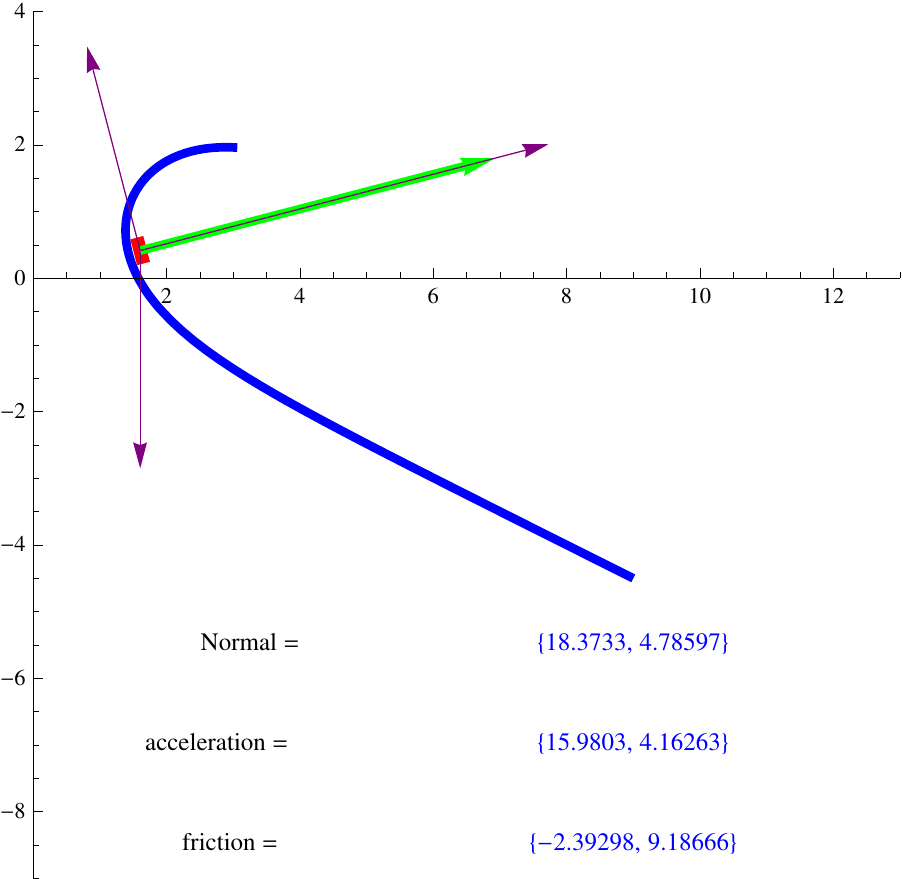}\includegraphics[width=4.6cm,height=4.6cm]{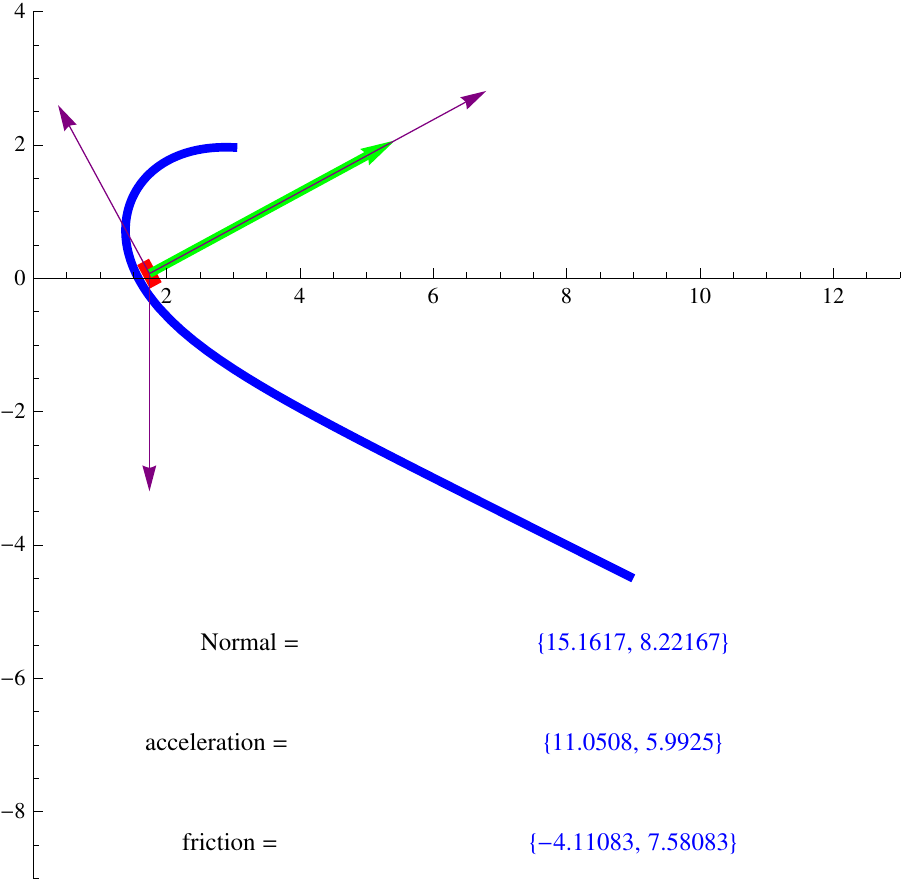}}
\end{figure}
\begin{figure}[ht]
\centerline{\includegraphics[width=4.6cm,height=4.6cm]{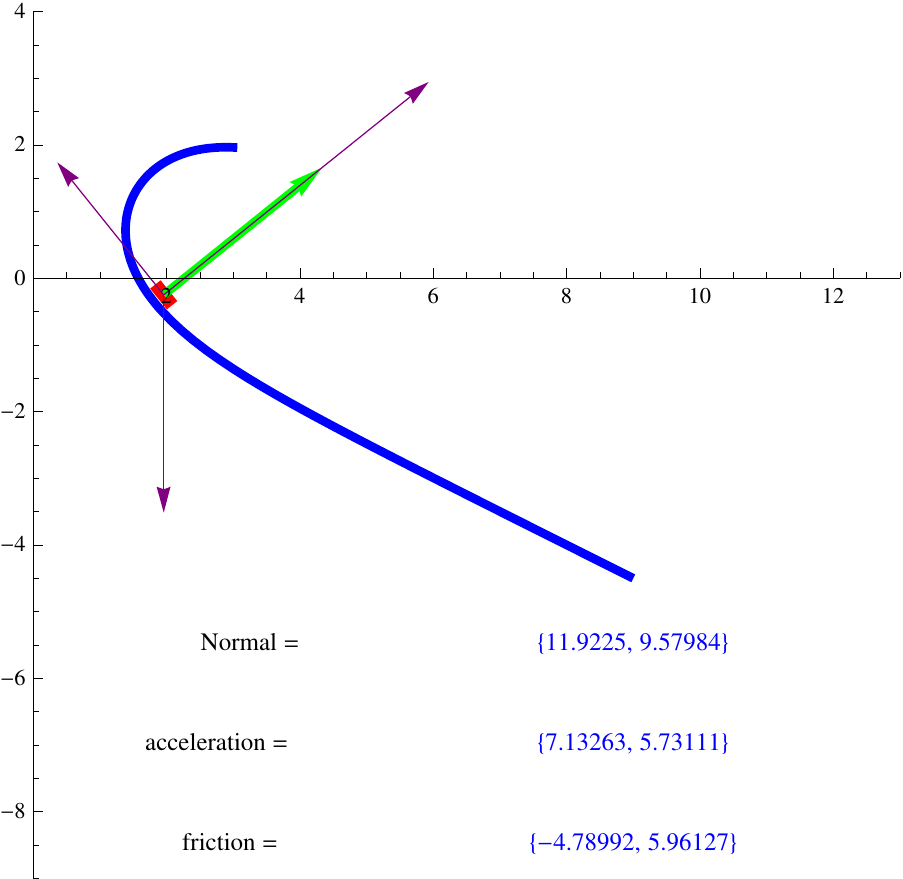}\includegraphics[width=4.6cm,height=4.6cm]{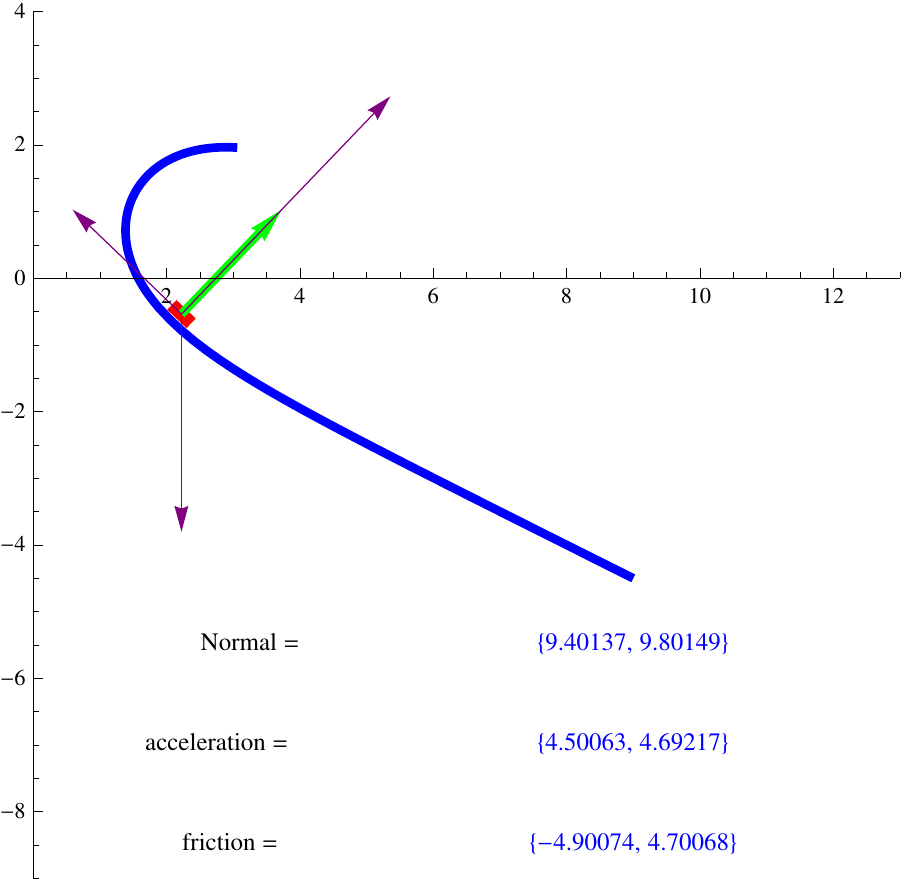}\includegraphics[width=4.6cm,height=4.6cm]{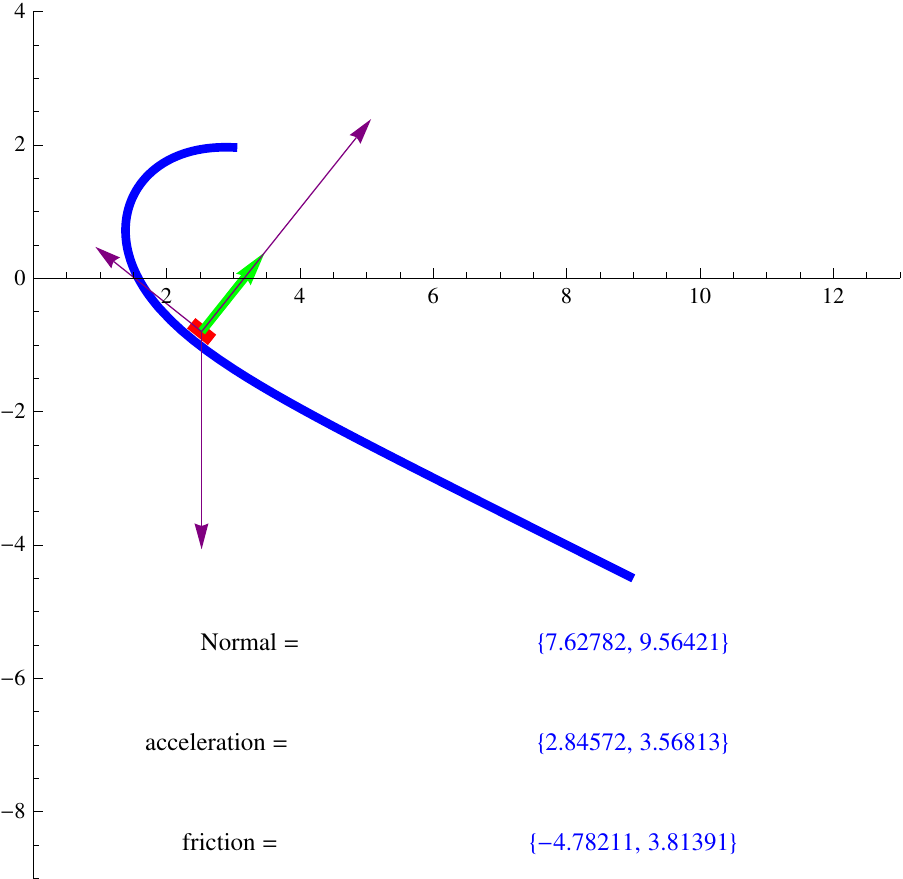}\includegraphics[width=4.6cm,height=4.6cm]{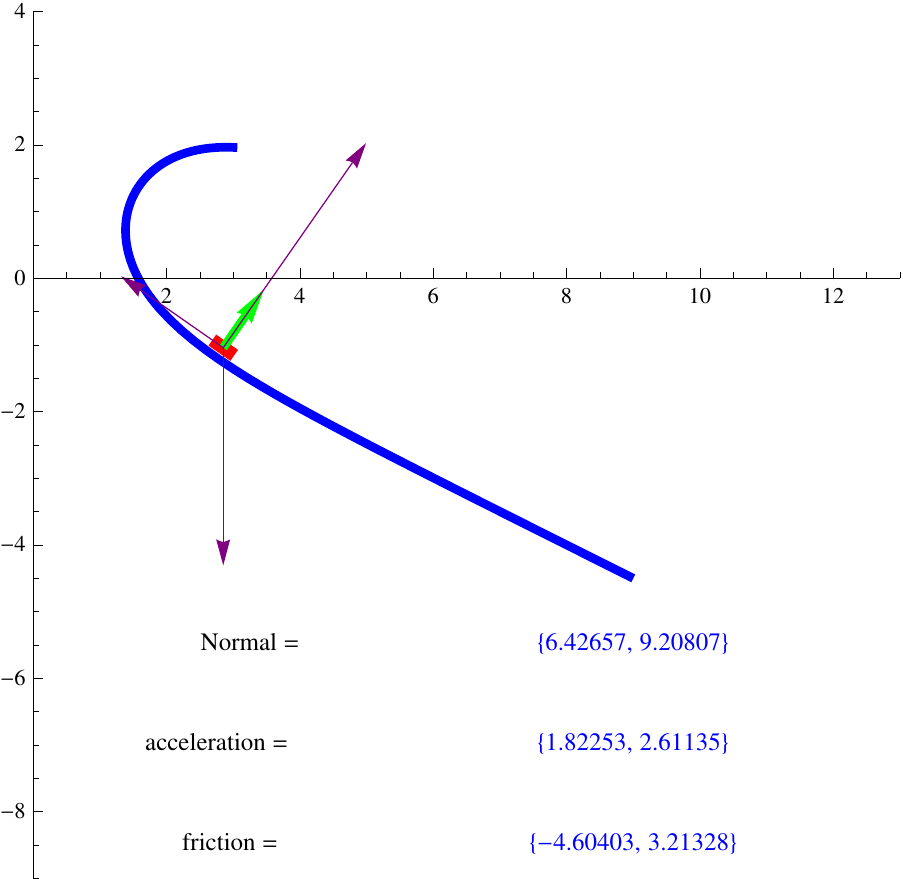}}
\end{figure}
\begin{figure}[ht]
\centerline{\includegraphics[width=4.6cm,height=4.6cm]{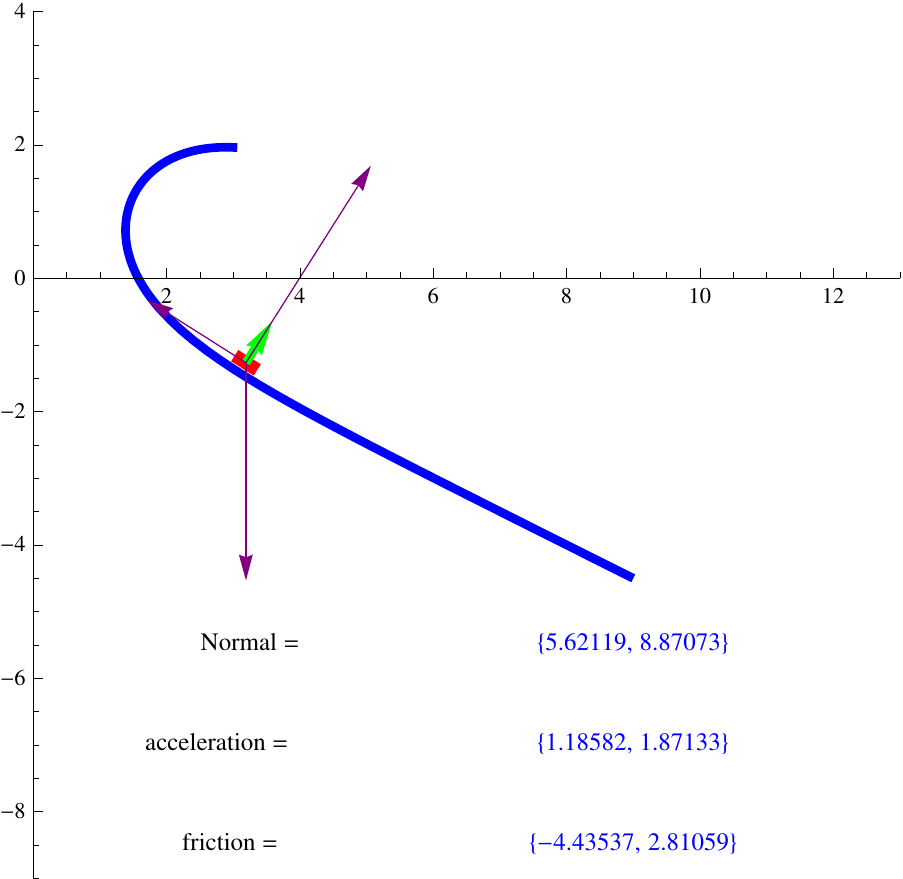}\includegraphics[width=4.6cm,height=4.6cm]{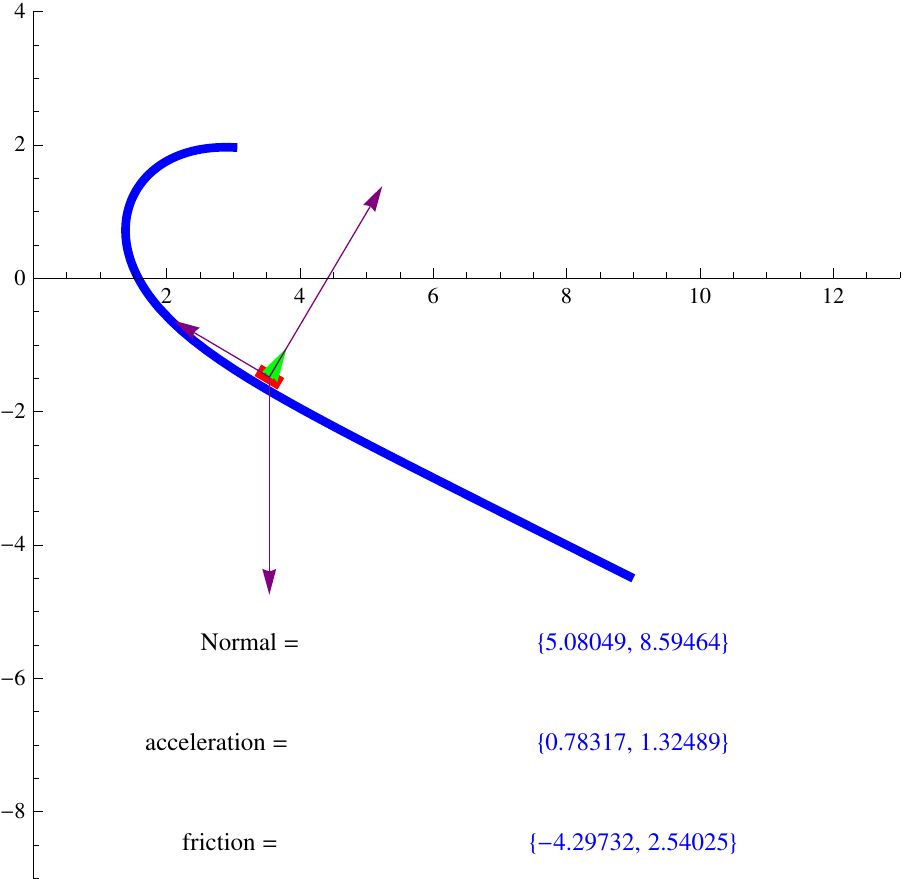}\includegraphics[width=4.6cm,height=4.6cm]{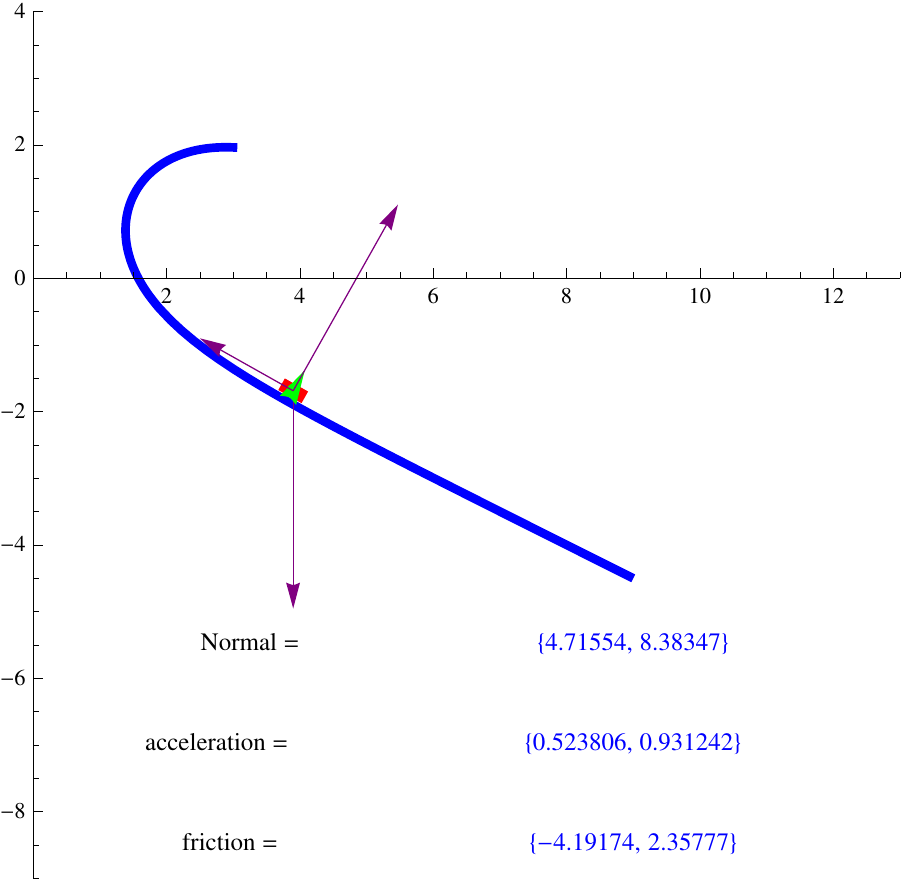}\includegraphics[width=4.6cm,height=4.6cm]{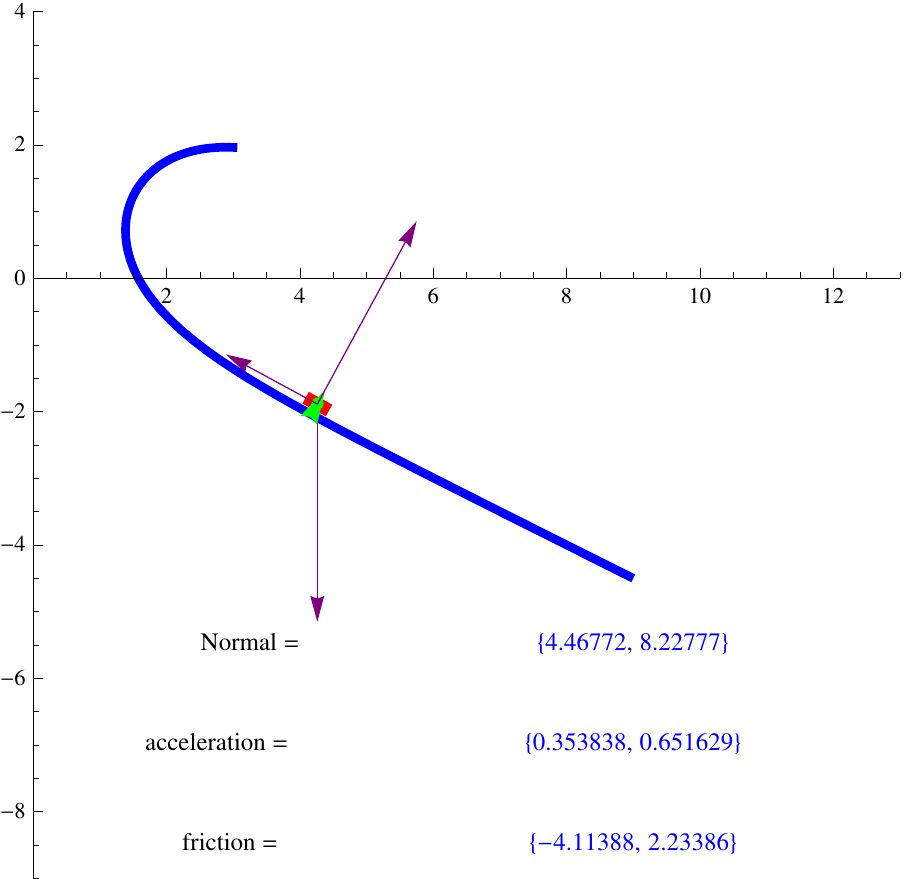}}
\end{figure}
\begin{figure}[ht]
\centerline{\includegraphics[width=4.6cm,height=4.6cm]{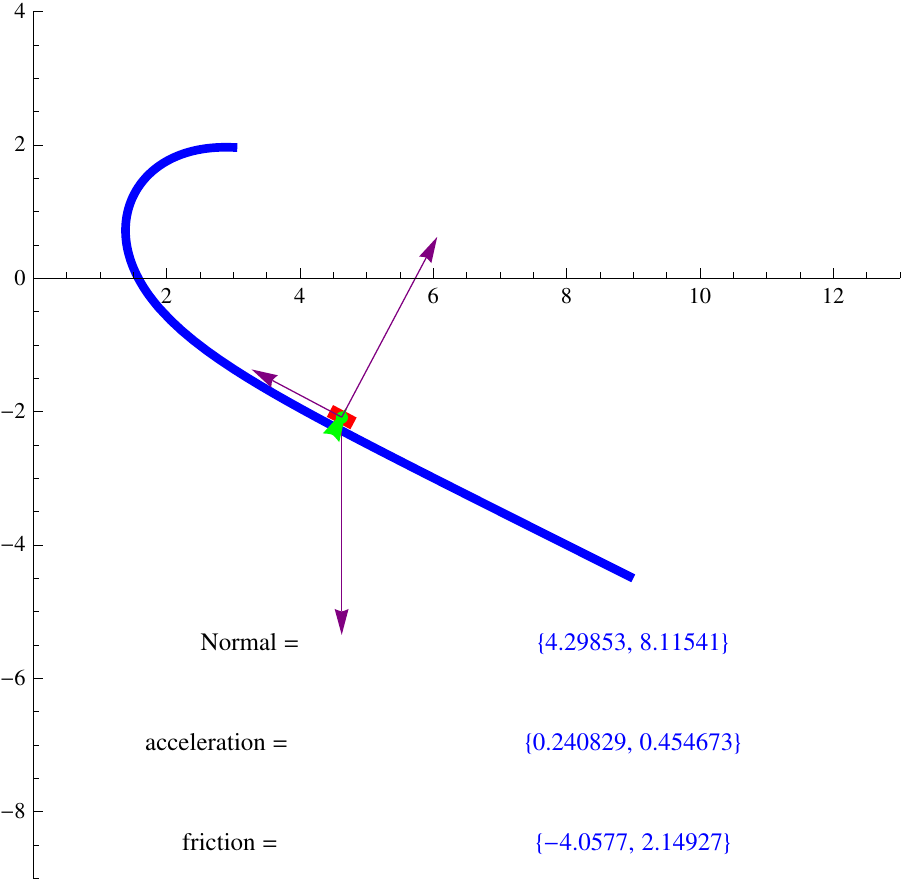}\includegraphics[width=4.6cm,height=4.6cm]{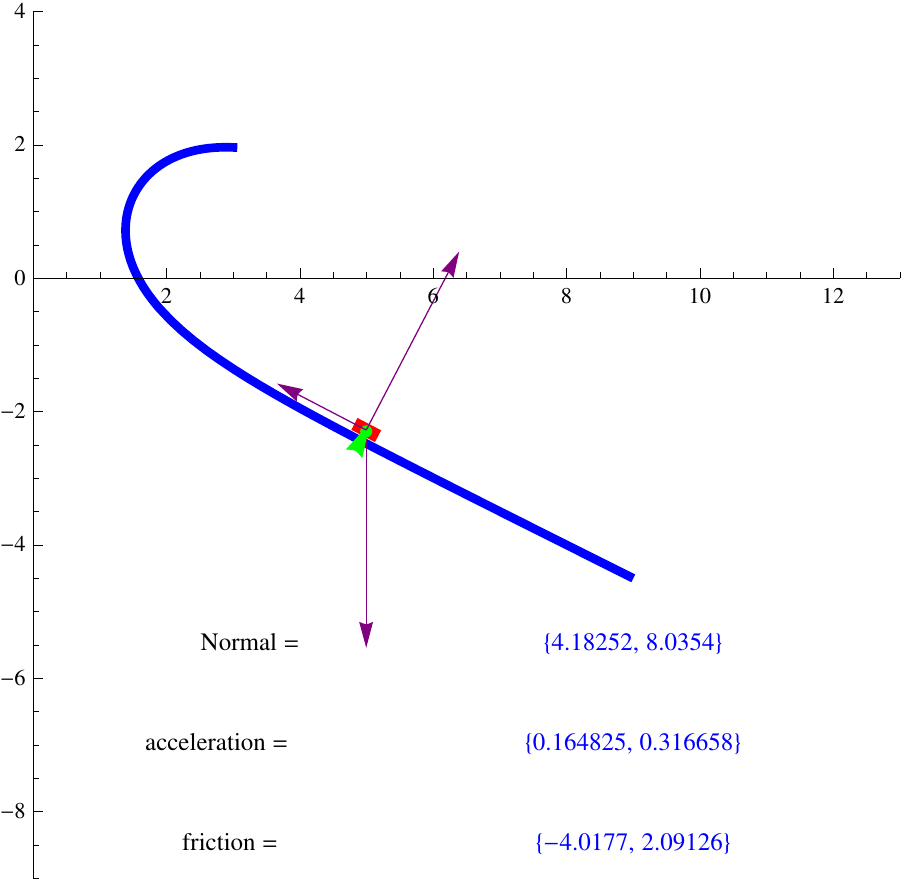}\includegraphics[width=4.6cm,height=4.6cm]{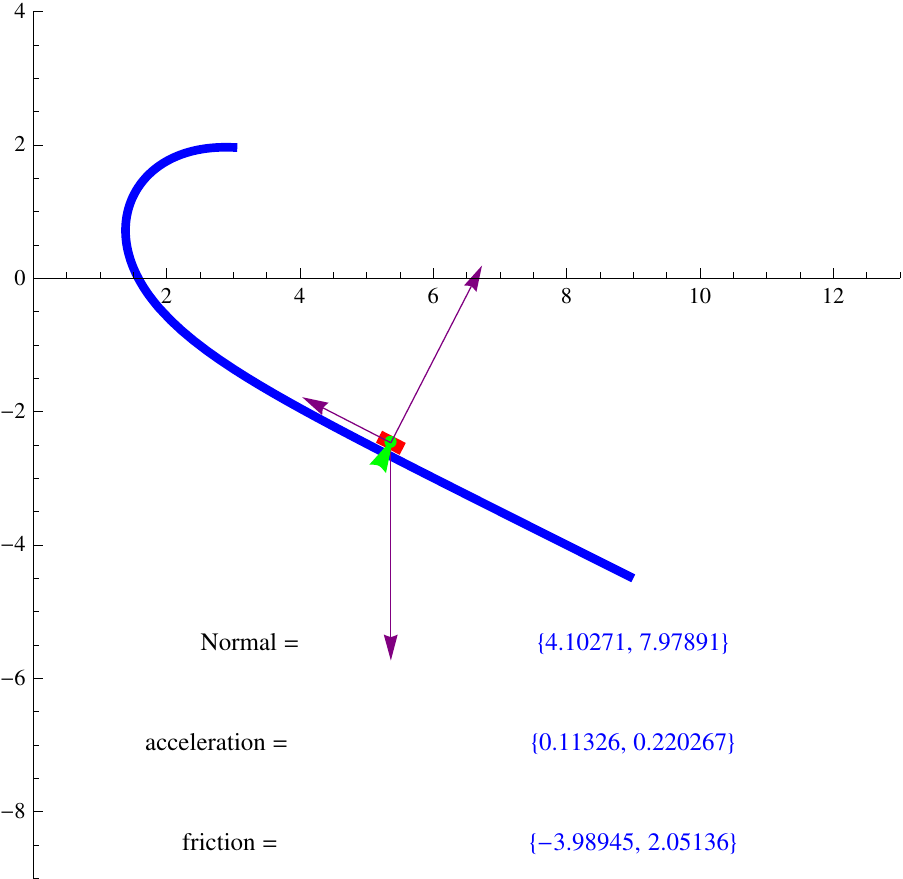}\includegraphics[width=4.6cm,height=4.6cm]{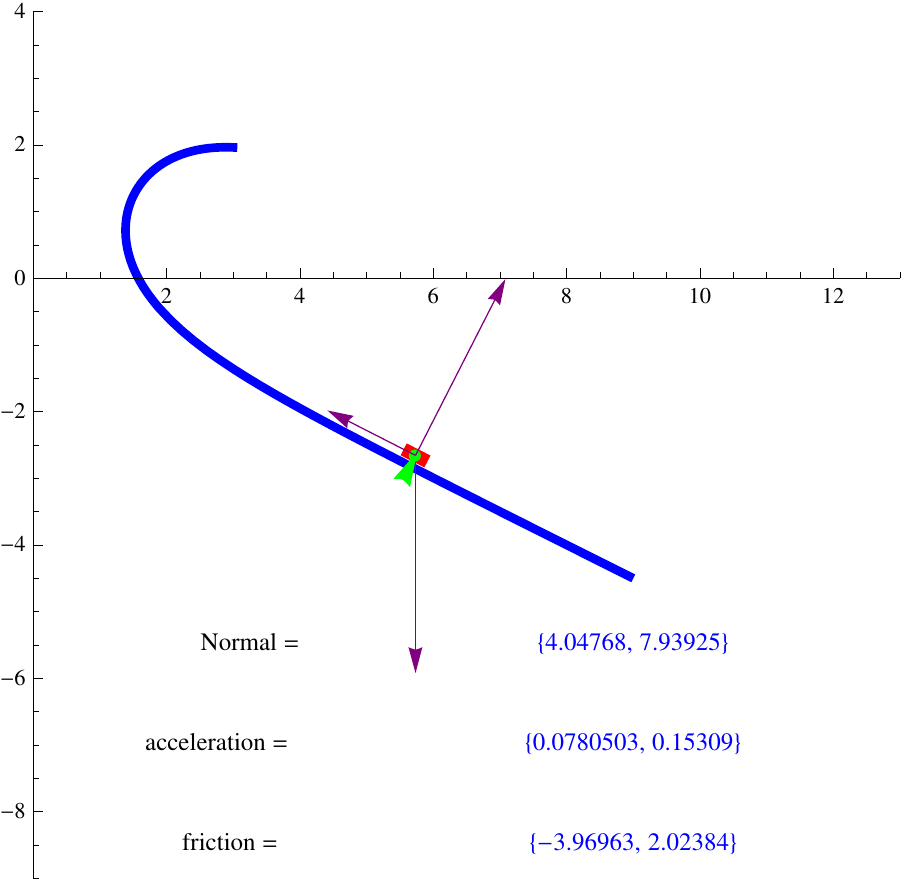}}
\caption{Motion on the longest parts of the ramps}
\label{lg}
\end{figure}

\begin{figure}[ht]
\centerline{\includegraphics[width=4.6cm,height=4.6cm]{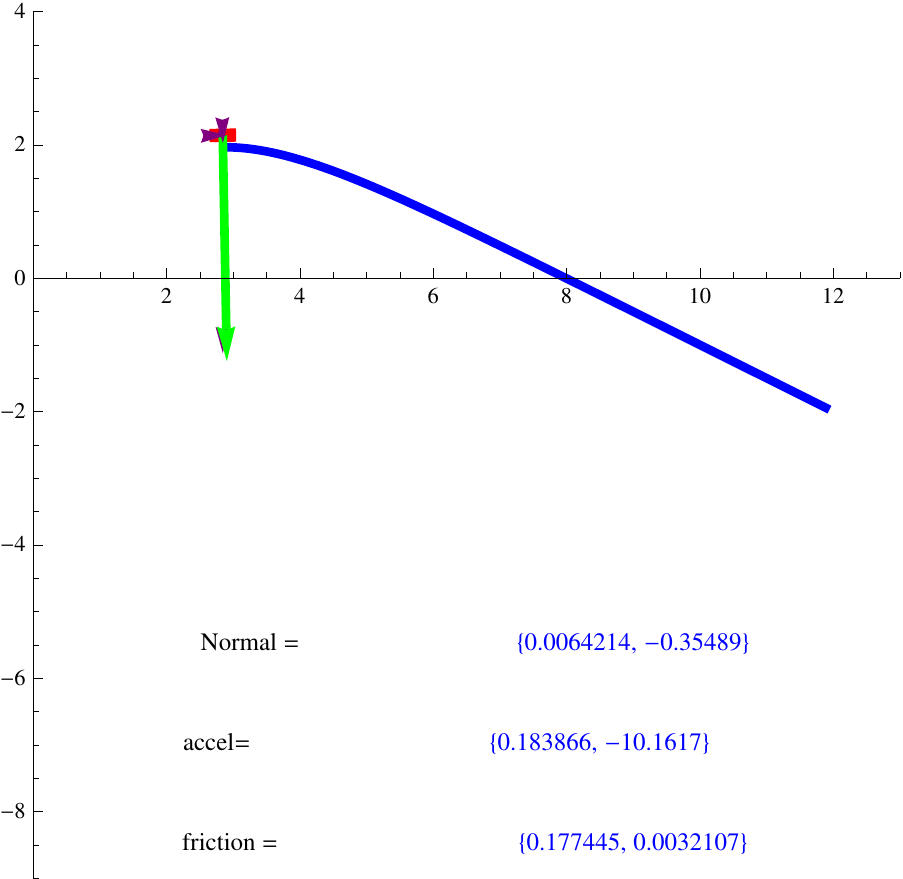}\includegraphics[width=4.6cm,height=4.6cm]{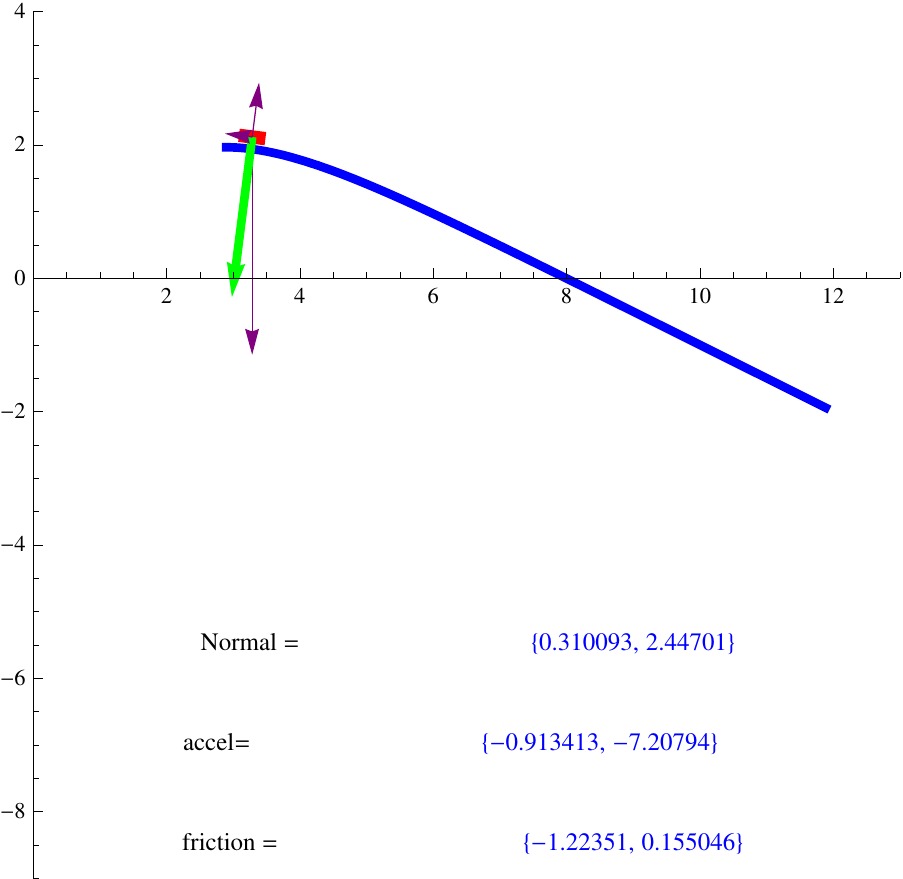}\includegraphics[width=4.6cm,height=4.6cm]{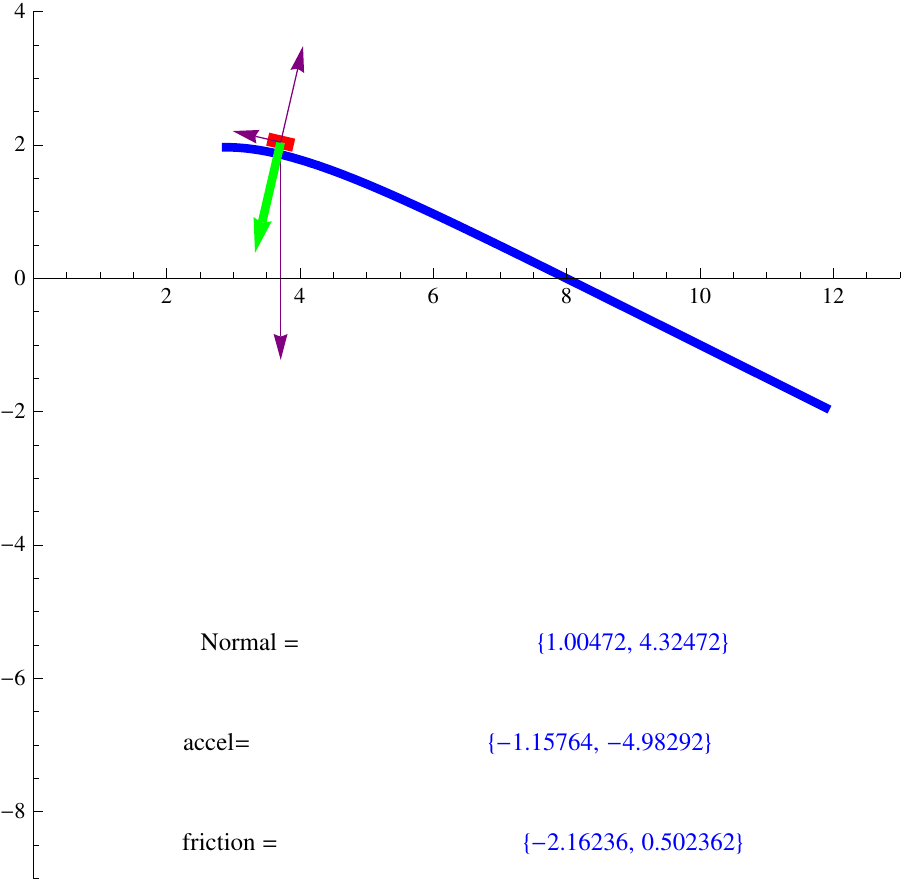}\includegraphics[width=4.6cm,height=4.6cm]{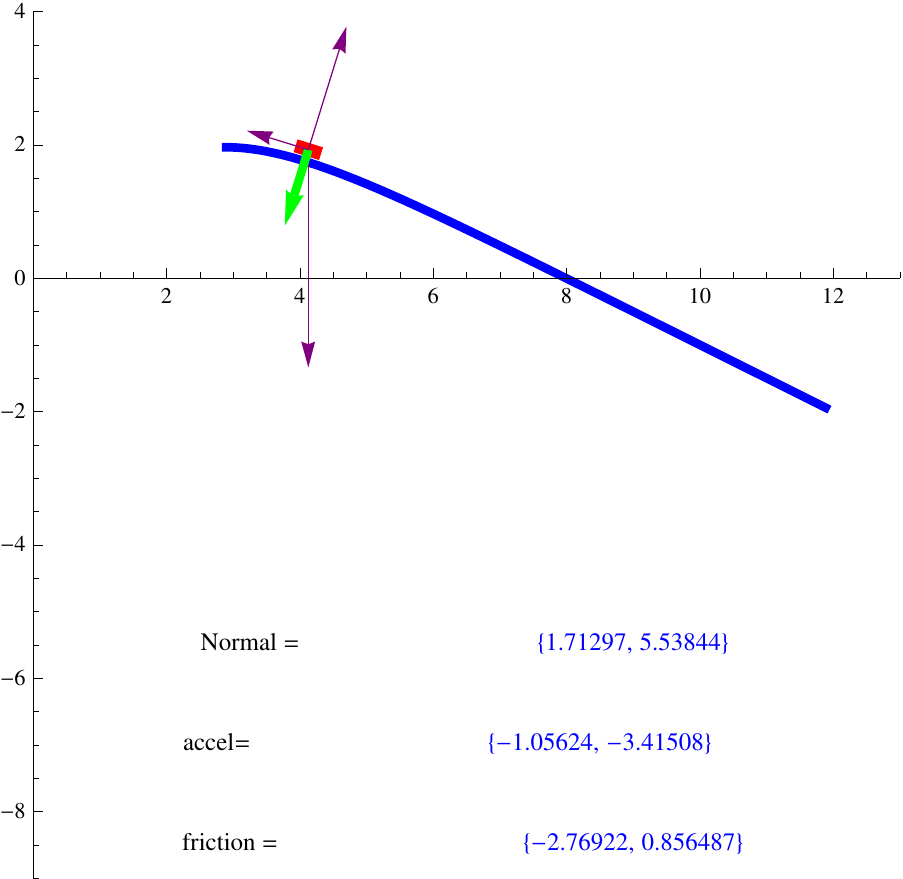}}
\end{figure}
\begin{figure}[ht]
\centerline{\includegraphics[width=4.6cm,height=4.6cm]{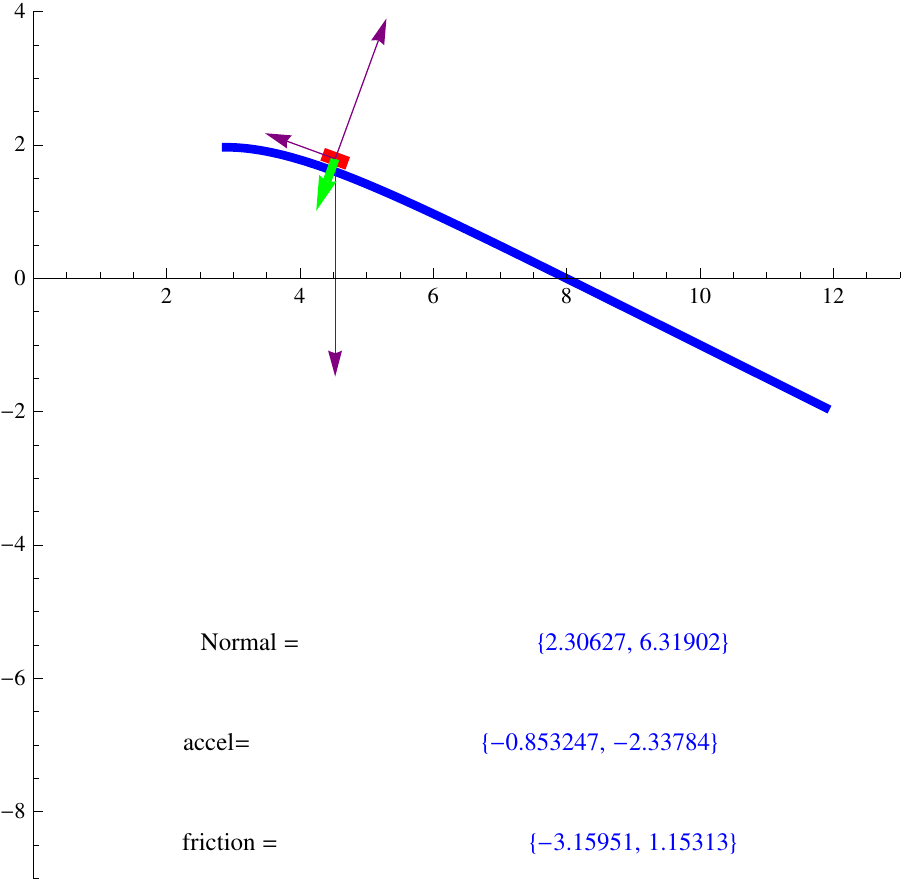}\includegraphics[width=4.6cm,height=4.6cm]{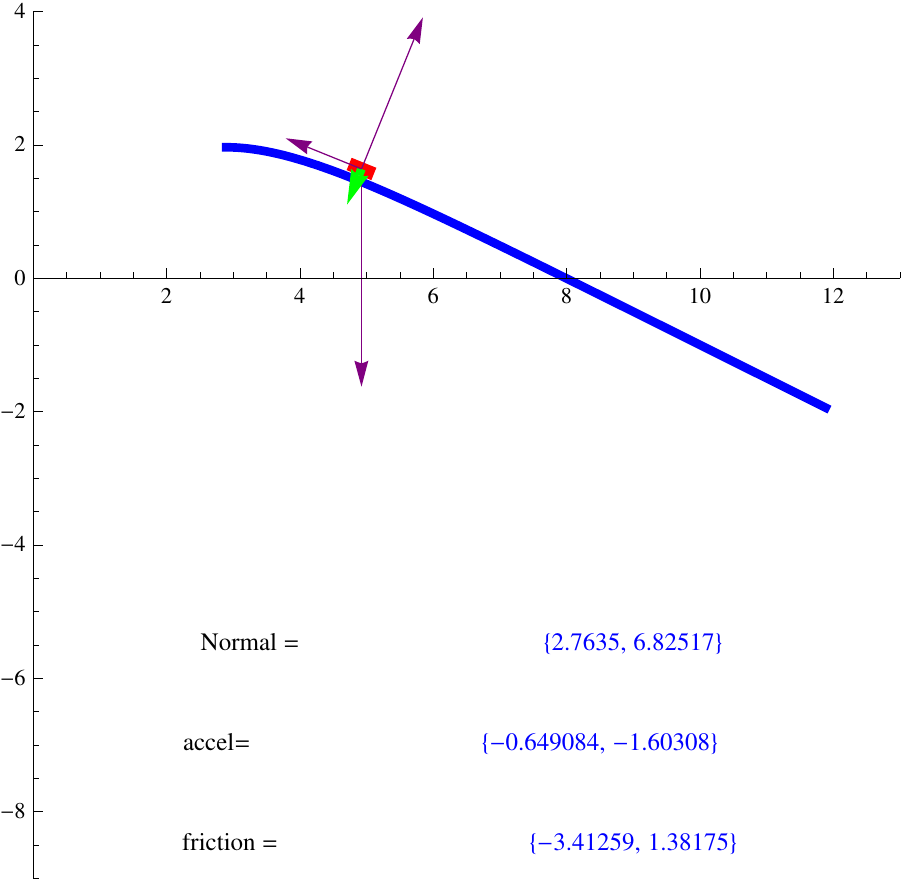}\includegraphics[width=4.6cm,height=4.6cm]{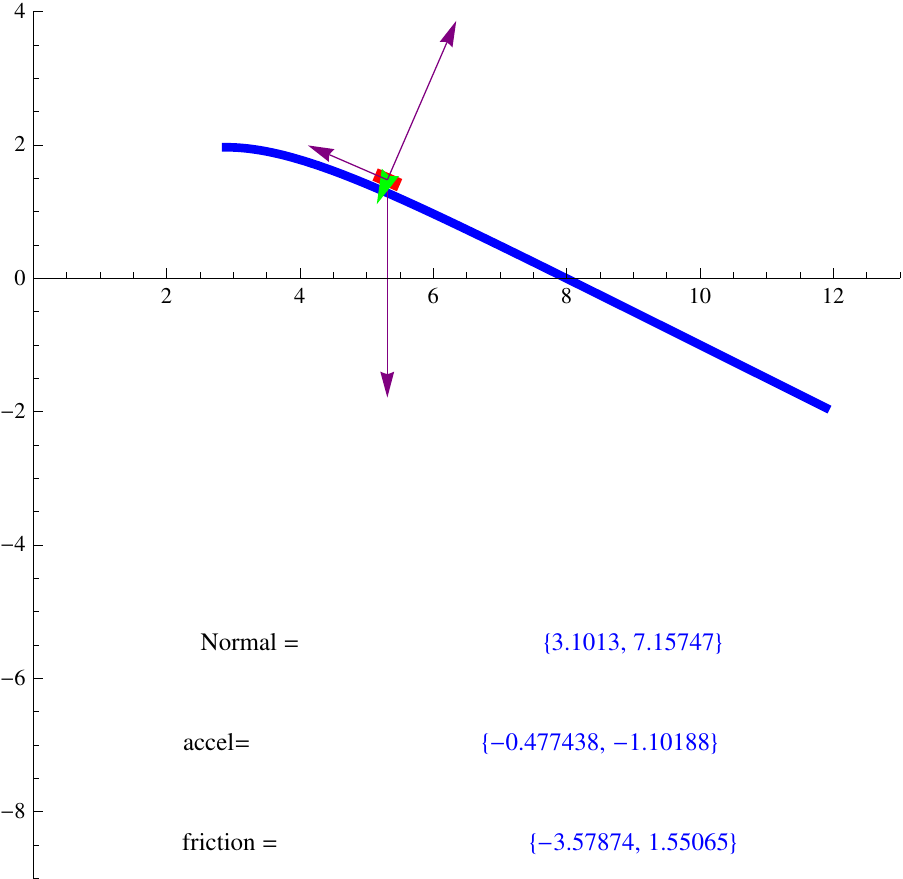}\includegraphics[width=4.6cm,height=4.6cm]{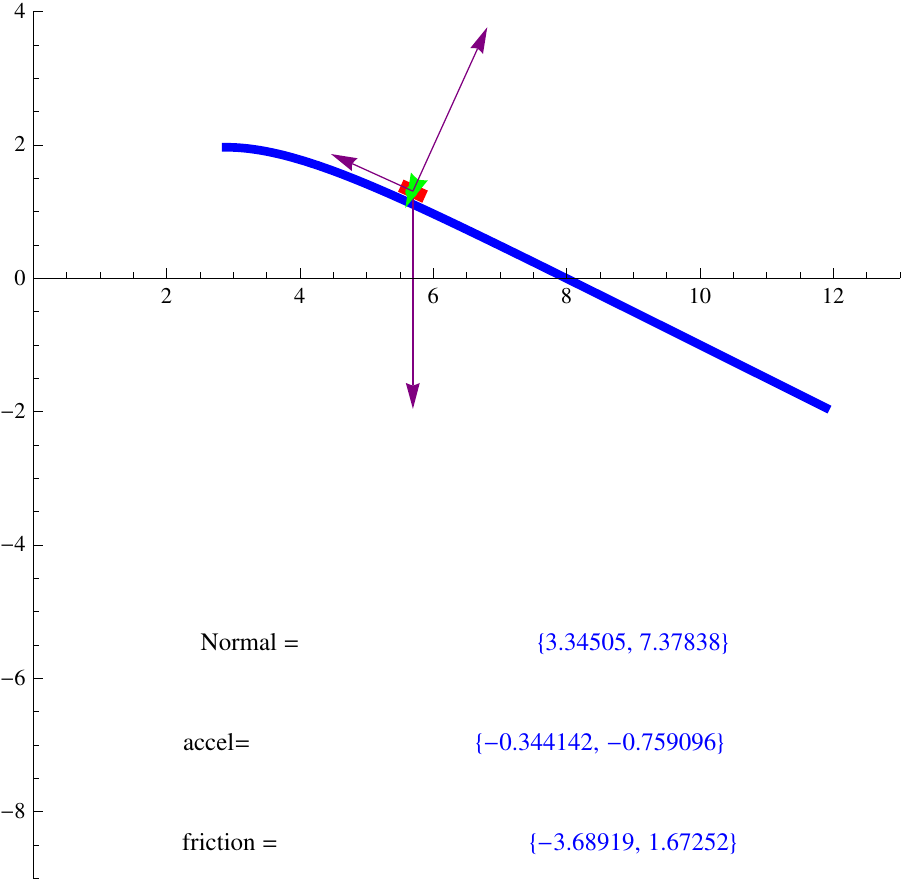}}
\end{figure}
\begin{figure}[ht]
\centerline{\includegraphics[width=4.6cm,height=4.6cm]{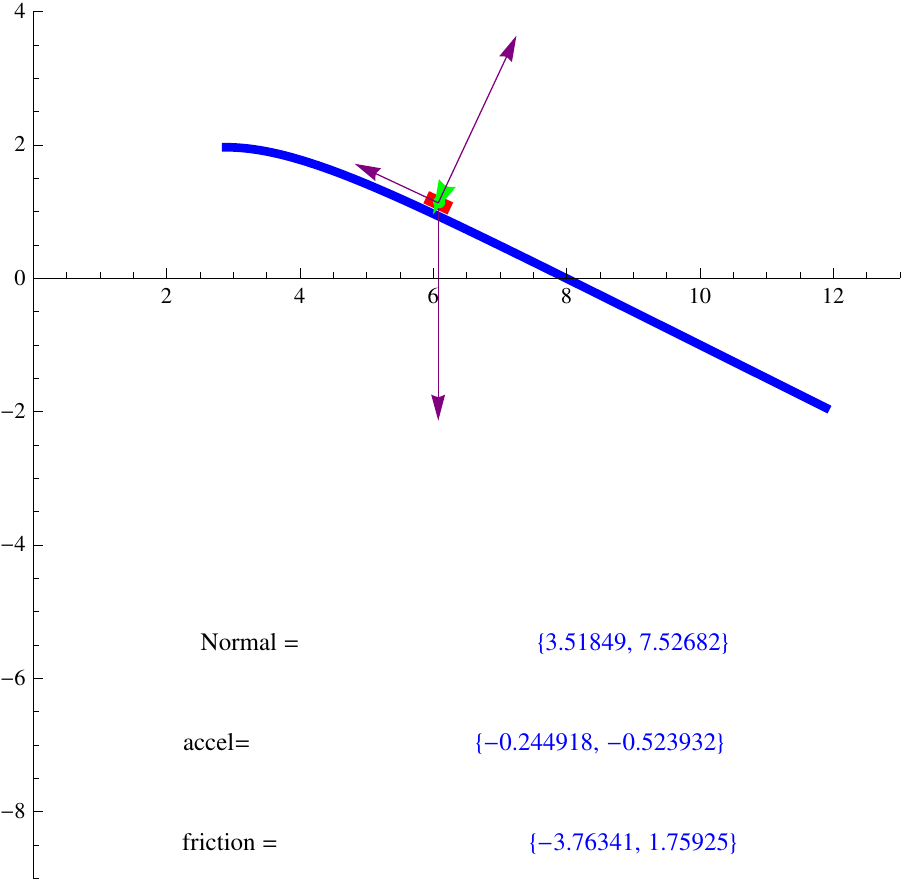}\includegraphics[width=4.6cm,height=4.6cm]{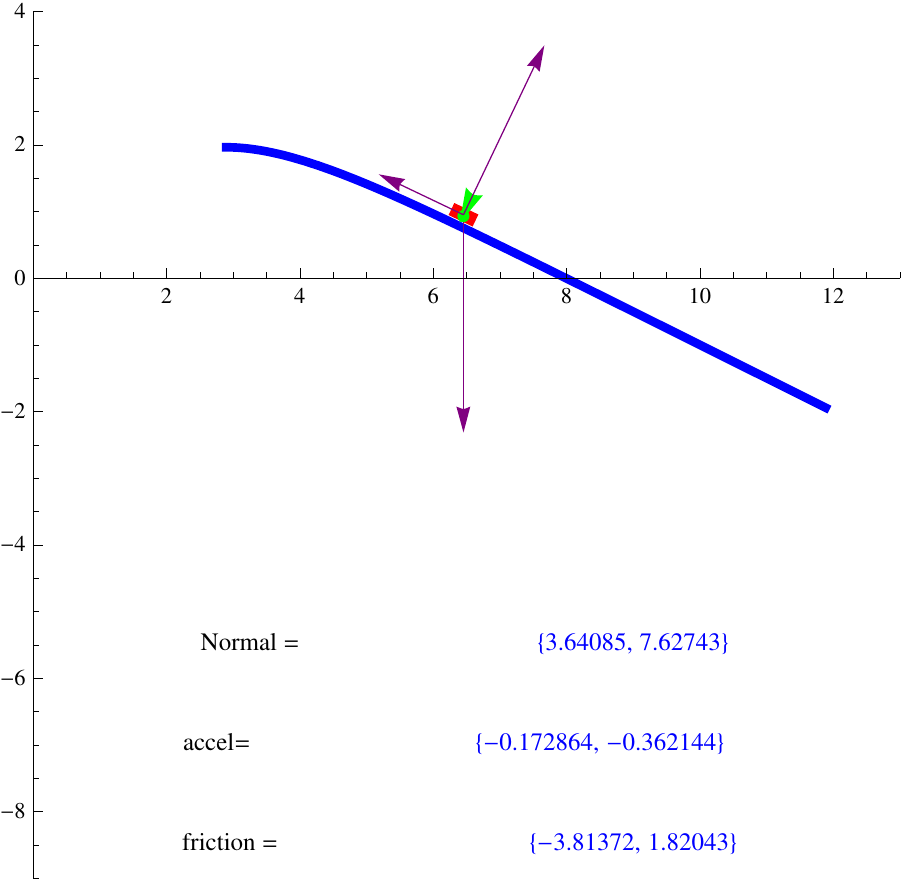}\includegraphics[width=4.6cm,height=4.6cm]{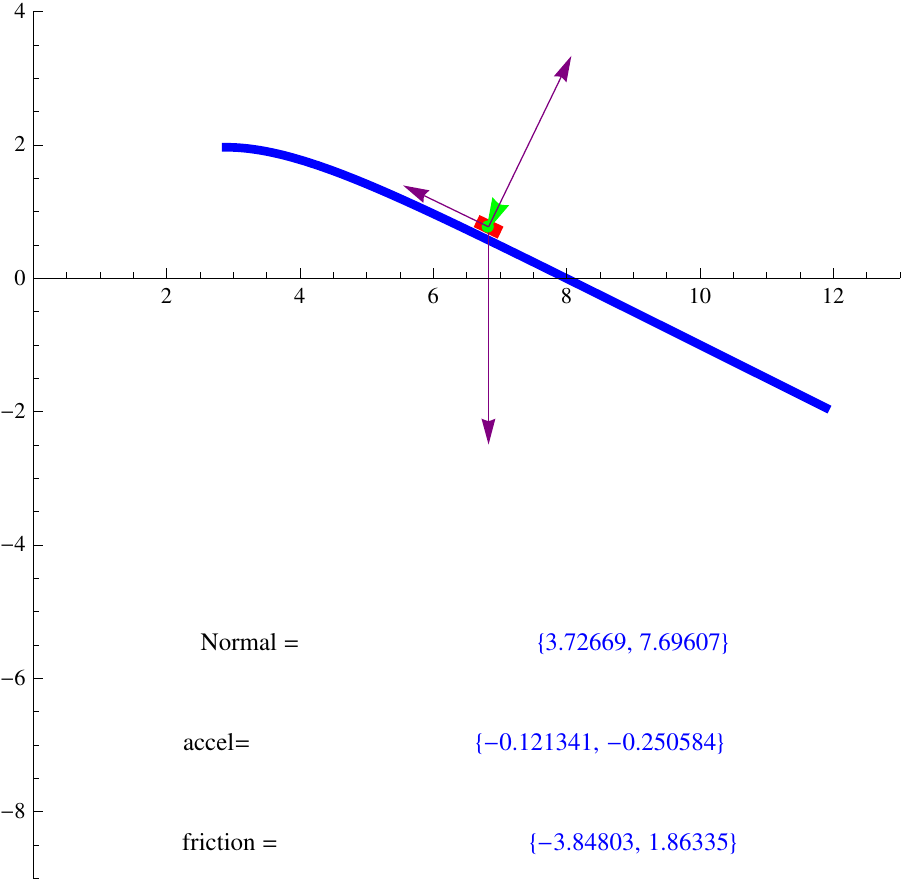}\includegraphics[width=4.6cm,height=4.6cm]{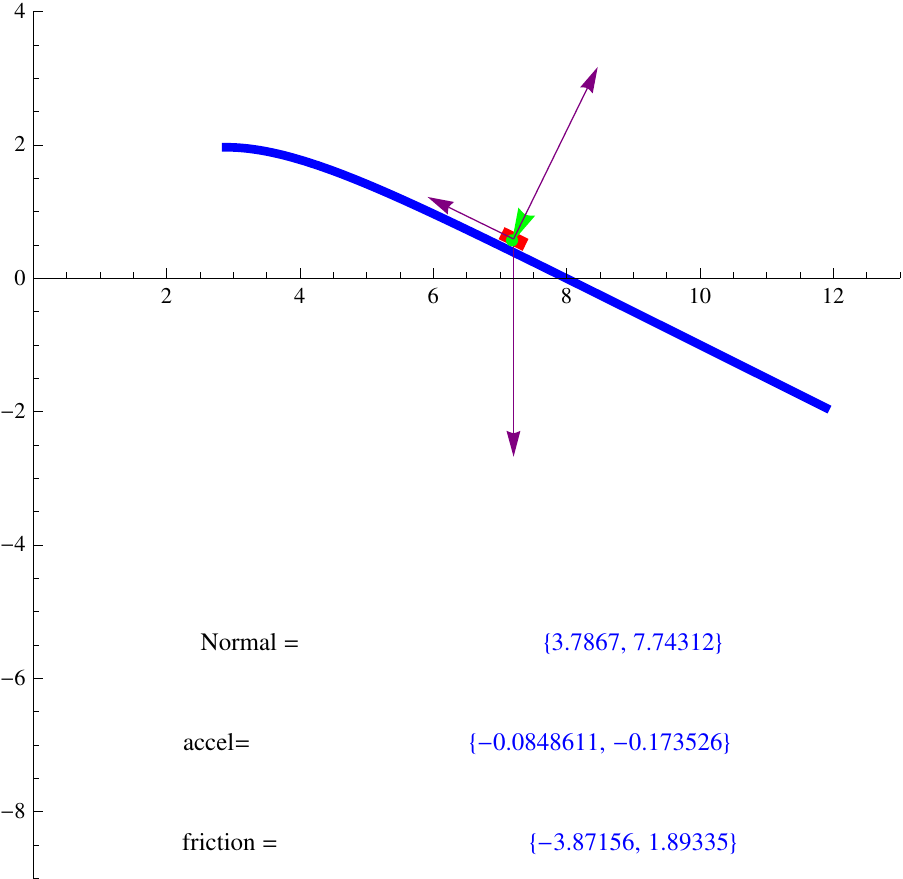}}
\caption{Motion on the shortest parts of the ramps}
\label{sg}
\end{figure}




\end{document}